\documentclass[11pt,a4wide]{article}

\usepackage{multirow}
\usepackage{algorithmic}
\usepackage{amsmath}
\usepackage{amsfonts}
\usepackage{amsthm}
\usepackage[english]{babel} 
\usepackage{dsfont}
\usepackage{amssymb}
\usepackage{graphicx}      
\usepackage{cite}                            % when including figure files
\usepackage{subcaption}
\usepackage{epstopdf}
\usepackage{epsfig}
\usepackage{subfloat}
\usepackage{float}
\usepackage[thinlines]{easytable}
\usepackage{array}
\usepackage{mathtools}
\usepackage{stmaryrd}
\usepackage{bm}
\usepackage{hyperref}
\hypersetup{
    colorlinks=true,
    linkcolor=blue,
    filecolor=magenta,      
    urlcolor=cyan,
    pdftitle={Overleaf Example},
    pdfpagemode=FullScreen,
    }
\usepackage[normalem]{ulem}

\usepackage{booktabs}
\usepackage{url}
\usepackage[font=small,labelfont=bf,textfont=it,indention=.1cm,width=1.1\textwidth]{caption}

\usepackage[capitalise]{cleveref}
\usepackage{comment}
\allowdisplaybreaks
\numberwithin{equation}{section}

\newtheorem{remark}{Remark}[section]

\newtheorem{lemma}{Lemma}[section]
\newtheorem{theorem}{Theorem}[section]

\newtheorem{proposition}{Proposition}[section]
\usepackage[linesnumbered,algoruled,boxed,lined]{algorithm2e}

\def\RR{\mathbb R}

\def\e{\epsilon}
\def\ve{\varepsilon}
\def\a{\alpha}
\def\b{\beta}

\def\argmax{{\rm arg}\!\max}
\def\be{\begin{equation}}
\def\ee{\end{equation}}
\def\bea{\begin{eqnarray}}
\def\eea{\end{eqnarray}}

\def \obj{j}
\def \pen{P}
\def \r {r}
\def \MM {\mathcal{M}}
\def \d {\text{d}}
\def\dist{\textup{dist}}

\newcommand{\lp}{\left(}
\newcommand{\rp}{\right)}

\DeclareMathOperator{\amin}{argmin}
\title{Constrained consensus-based optimization}

\author{Giacomo Borghi\footnote{RWTH Aachen University, Institute for Geometry and Applied Mathematics, Aachen, Germany; University of Ferrara, Department of Mathematics and Computer Science, Ferrara, Italy (borghi@eddy.rwth-aachen.de)}
\and Michael Herty \footnote{RWTH Aachen University, Institute for Geometry and Applied Mathematics, Aachen, Germany (herty@igpm.rwth-aachen.de)}
\and Lorenzo Pareschi\footnote{University of Ferrara, Department of Mathematics and Computer Science, Ferrara, Italy (lorenzo.pareschi@unife.it)}}

\begin{document}
\maketitle

\begin{abstract} In this work we are interested in the construction of numerical methods for high dimensional constrained nonlinear optimization problems by particle-based gradient-free techniques. 
A consensus-based optimization (CBO) approach combined with suitable penalization techniques is introduced for this purpose.  
The method relies on a reformulation of the constrained minimization problem in an unconstrained problem for a penalty function and extends to the constrained settings the class of CBO methods. 
Exact penalization is employed and, since the optimal penalty parameter is unknown, an iterative strategy is proposed that successively updates the parameter based on the constrained violation.
Using a mean-field  description of the the many particle limit of the arising CBO dynamics, we are able to show convergence of the proposed method to the minimum for general nonlinear constrained problems. Properties of the new algorithm are analyzed. Several numerical examples, also in high dimensions, illustrate the theoretical findings and the good performance of the new numerical method.
\end{abstract}

{\bf Keywords}: consensus-based optimization, constrained nonlinear minimization, gradient-free methods, mean-field limit  

\tableofcontents

\section{Introduction}

Consensus-based optimization (CBO) methods have been introduced and studied recently as efficient computational tools for solving high dimensional nonlinear unconstrained minimization problems^^>\cite{carrillo2021consensus,jin2020convergence,carrillo2018analytical,carrillo2019consensus,totzeck2018numerical,gp20,benfenati2021binary,pinnau2017consensus,jinew20,huang2021meanfield,totzeck2021trends}. They belong to the family of individual-based models that are inspired by self-organized dynamics based on alignment^^>\cite{MR3274797,defrli13,Prigogine1977self,Vicseck}. In the context of optimization, CBO methods can be viewed as continuous versions of particle-based metaheuristic techniques or evolutionary algorithms. Indeed, in such methods, an ensemble of interacting particles explores the landscape of the objective function through a combination of alignment towards a collective estimate of the global minimum and stochastic exploration with noise proportional to the distance from that minimum. One of the key aspects of CBO methods is the use of a convex combination of the values explored by the particles in estimating the global minimum according to the Laplace principle^^>\cite{pinnau2017consensus}. 

The recent popularity of CBO type methods stems, in particular, from their simple first order structure and the fact that, for high--dimensional, non convex, nonlinear unconstrained problems, numerical and analytical evidence of their performance has been reported (see^^>\cite{jin2020convergence,carrillo2019consensus} for machine learning applications). Related lines of research have considered mean-field descriptions of particle swarm optimization strategies^^>\cite{gp20,hui20} and
nonlinear sampling techniques, like Ensemble Kalman filtering^^>\cite{MR4121318,MR4059375}.
For an exhaustive list of references as well as a review of existing recent results, we refer to the recent survey articles^^>\cite{carrillo2021consensus,totzeck2021trends,PSOsurvey}. 

 Currently, most of the presented results are so far for the {\em unconstrained} case. There are few results on constrained optimization problems. In^^>\cite{fhps20-3,fhps20-2,fhps20-1} these methods have been considered constrained to hypersurfaces, specifically analyzing the case of the sphere given its importance in numerous applications and demonstrating convergence to the global minimum. Here, we present a method for general constrained problems, that rely on a reformulation of the constrained minimization problem in an unconstrained problem for a penalty function. Since the exact penalty parameter is unknown, 
an iterative strategy is proposed that  successively updates the parameter based on the constrained
violation.

Related algorithms for finite--dimensional
quadratic programming problems have been introduced in ^^>\cite{spellucci2002,herty2007smoothed}. 
Therein, however, additional smoothing of the penalty parameter 
has been conducted to apply second-order methods. The method^^>\cite{spellucci2002,herty2007smoothed} are based on
smooth approximations to the exact $\ell_1$--penalty function which extends other results as e.g.^^>\cite{Zangwill1967,HanMangasarian1979,bertsekas1975,PilloGrippo1989,CG}.

In the nonlinear optimization context,  special methods dealing with the
non--differentiability have been proposed (see e.g.
\cite{Conn1973,MayneMaratos1979,HintermullerUlbrich2004} or 
\cite{Petrzykowski1969,ChenMangasarian1996,BentalTelboulle1988}).
 Here, we do not require any additional smoothing or any special treatment of the non--differentiability, since CBO does not rely on explicit gradient information. 
At the same time as finalizing this manuscript, we learned that the case of smooth penalization in the CBO context has been discussed in^^>\cite{carrillo2021consensusbased}. In their approach however the penalty parameter needs to tend to infinity and contrary to our work no explicit update rule of the parameter is given. 
 \par 
Even so, the optimization problem itself is finite--dimensional with $x \in \mathbb{R}^d,$  the presented method relies on an infinite--dimensional reformulation
to derive the convergence results as outlined above. This is different from a penalty method in infinite--dimensions as e.g. presented in ^^>\cite{MR3028186,MR2724157}. 
\par

The rest of the manuscript is organized a s follows. In Section 2 we introduce the penalization approach for the constrained minimization problem and the corresponding CBO method with adaptive strategy of the penalization parameter. Next, in Section 3 we analyze the convergence properties of the method. To this aim we introduce the corresponding mean-field approximation and show that, under suitable assumptions, the method converges to the minimum for general nonlinear constrained problems. Several numerical results that illustrate the previous analysis and the performance of the method are then reported in Section 4. The manuscript ends with some final conclusions in the last section.

\section{Consensus-based methods for constrained minimization problems}

We are interested in solving constrained optimization problems of the type
\be
\min_{x \in \RR^d} \obj(x) \quad \text{subject to} \quad x \in \MM\,,
\label{pb}
\ee
where the cost functional  $j\in \mathcal{C}(\RR^d, \RR)$ is continuous. The feasible set $\MM \subset \RR^d$ is assumed to have a boundary of zero Lebesgue-measure, i.e.,  
 $\d x(\partial M) = 0$, but we do not require $\MM$ to be convex, contrary to 
 ^^>\cite{ bae2021constrained}.
 \par 
This assumption is not restrictive and as an example,  $\MM = \{ x: g(x) \leq 0 \}$  defined by inequality constraints, where $g(x) = G^\top x + g_0$ for some $G \in \RR^{n \times m}, g_0 \in \RR^m$ has this property.
\par 
The constrained problem is reformulated as unconstrained problem for an exact penalty function $\pen(x,\b)$ depending on the (unknown) multiplier $\b \in \RR.$ In the following, the  exact penalization of the objective function is used 
\be
\notag
 \pen(x,\b):= \obj(x) + \b\r(x) \qquad \text{where} \quad \b \geq 0, \quad \r(x)
\begin{cases} 
=0\ & \text{if}\quad x \in \MM \\
>0\ & \text{else}
\end{cases}
\ee

\be
\min_{x \in \RR^d} P(x, \b)\,.
\label{pb:pen}
\ee

We denote by \eqref{pb:pen} the penalty subproblem and by $P (x, \b)$ the penalty function. The penalty term depends on the  parameter $\b$ and we assume that 

\begin{enumerate}
\item[(A1.1)] There exists $\bar \b\geq 0$, such that for all $\b \geq \bar \b$, the minimum of $\pen(x,\b)$ is the global solution of \eqref{pb}. 
\item[(A1.2)] For $\b < \bar \b$ the two problems are not equivalent, namely, $\inf_{x \in \RR^d} P(x, \b) < \inf_{x \in \MM} j(x)$.
\end{enumerate}
Note that (A1.2) is a technical assumption while (A1.1) states that $P$ is an exact penalty function also for the global minimum. Exact penalization has been investigated in many publications and we refer to ^^>\cite{bertsekas1975,HanMangasarian1979,bertsekas1982,  burke1991} for more details on this topic.
 
 Let us only note here, that for a generic problem where $\MM:=\{ x \in \RR^d: g(x)\leq 0\}$ and $r(x)= \| \max(0,g(x)) \|_1$
 with $j,g$ twice differentiable at the global minimum $x^*$ of \eqref{pb}, if the KKT conditions and the weak second-order sufficient optimality condition hold, a value $\bar \beta$ which satisfies (A1.1) is given by any $\bar\beta > \|\lambda\|_\infty$ where $\lambda$ is the Lagrange multiplier to $g$^^>\cite{bonnans2013numerical}.

 In the case of exact penalization, it is also known, that the penalty function $r(x)$ is  non--differentiable, as for instance in the case of the exact $\ell_1$--penalization $r(x) =  \| \max(0,g(x)) \|_1$.

\par 
Therefore, in the following we do not assume differentiability properties of $P.$ 
Note that, smooth penalization approaches, such as $\ell_2$--penalization $r(x) =  \| \max(0,g(x)) \|^2_2$, would allow us to preserve the (eventual) smoothness of the objective function at cost of taking a possibly unbounded penalty parameter $\b \to \infty.$ In this case, the solution of \eqref{pb:pen} for a given (finite) $\b$, may be infeasibile for every $\b\geq0$ and may converge the solution of \eqref{pb} only asymptotically as $\b \to \infty$, see^^>\cite{burke1989}.

\subsection{Constrained CBO methods}
%%% 
We assume (A1) and propose a modified CBO method to solve \eqref{pb:pen}. To simplify the description we consider the case where the corresponding system of stochastic differential equations is solved by the Euler-Maruyama method^^>\cite{pinnau2017consensus,carrillo2019consensus}.
Starting from an initial set of particles $(X_0^i)_{i=1}^N$ sampled from a common given  distribution  $f_0$, $X_0^i \sim f_0$,  the CBO scheme iteratively updates the particles position to explore the objective function landscape and, eventually, concentrate around the (global) minimizer.
\par 
Before we describe the iterative step and the application to the penalty problem, we first introduce the weighted average $X^\a_k$ which is calculated through a Gibbs-type distribution dependent on the penalty function $P(x, \beta)$,

\begin{equation}
X_k^\a = \frac{1}{Z_\alpha}\sum_{i=1}^N X_k^i \exp\left(-\alpha \pen(X_k^i,\beta)\right), 
\label{eq:Xa}
\end{equation}
where $Z_\alpha$ is a normalization constant. The expectation of $X_k^\a$ can then be seen as an approximation of $$\amin_{i=1,\dots,N} P(X_k^i,\beta),$$ since 
\be
X^\a_k \; \rightarrow \;\amin_{i=1,\dots,N} P(X_k^i,\beta) \quad \text{as} \quad \a \to \infty\,,
\notag
\ee
provided that the global minimum exists^^>\cite{pinnau2017consensus,carrillo2018analytical,carrillo2019consensus}. 
Also, the choice of the Gibbs distribution in the definition of $X^k_\alpha$ is justified by the Laplace principle^^>\cite{Dembo2010} which states that for any absolutely continuous distribution $f \in \mathcal{P}(\RR^d)$ we have 
\be
\lim_{\a \to \infty}\left( - \frac 1\a \log \left( \int  e^{-\alpha P(x, \beta)} \d f(x) \right) \right) = \inf_{x \in \text{supp}(f)} P(x, \b)\,.
\label{eq:laplace}
\ee

\par
In a CBO method, at every step $k$, the particles are driven towards $X^\alpha_k$. In this way, we expect that for large values of $\a$ the system tends to concentrate among the particles where $P$ attains its minimum. In  addition,  we add a stochastic component in the state update, which depends on a set of
vectors $(B_k^i)_{i=1}^N, B_k^i \in \RR^d$ sampled from the normal standard distribution and on a given matrix $D^i_k \in \RR^{d \times d}$.

That is, the iteration reads
\be
X_{k+1}^i =X_{k}^i  -\lambda(X_k^i - X_k^\a)\Delta t + \sigma D^i_k B_k^i\sqrt{\Delta t}
\label{eq:iter}
\ee

where $\Delta t>0$ and the two parameters, $\lambda>0$ and $\sigma>0$, control the drift towards  $X_k^\a$ and the influence of the stochastic component respectively. 

The choice of $D_k^i$ characterizes the particles stochastic exploration process. As an example, the isotropic exploration was introduced in^^>\cite{pinnau2017consensus} and reads
\be
D_{k,\text{iso}}^i =  \| X_k^i - X^\a_k\| I_d\,,
\label{eq:iso}
\ee
where $I_d$ is the $d$-dimensional identity matrix, whereas in the anisotropic exploration, introduced in^^>\cite{carrillo2019consensus}, we have that
\be
D_{k,\text{aniso}}^i =  \text{diag} \left( (X_k^i - X^\a_k)_1 , \dots ,(X_k^i - X^\a_k)_d  \right)\,.
\label{eq:aniso}
\ee

For both processes, the magnitude of the random component associated with the particle $X^i_k$ depends on the difference between the weighted average $X^\a_k$ and the particle itself. In particular, particles that are far from $X^\a_k$ have a stronger exploration behavior compared to those close to it. The difference between the methods \eqref{eq:iso} and \eqref{eq:aniso} is on the direction of the stochastic component. Indeed, while in the isotropic case every dimension is equally explored, in the anisotropic process the particles explore each dimension at a different rate. The component-wise exploration better suits high dimensional problems, as the particles convergence rate is independent of the dimension $d$^^>\cite{carrillo2019consensus,fornasier2021convergence}. While we will focus the convergence analysis on the  isotropic CBO method, we will discuss extensions to the anisotropic case and compare the isotropic and anisotropic exploration processes on numerical examples  in Section \ref{s:43}.

\subsection{The update strategy for the penalty parameter}
 Since consensus-based optimization methods may handle non--differentiable objective functions without any additional effort, we use the concept of exact penalization. However, the value of $\bar \b$ such that (A1) holds is not known.
\par 
For any fixed parameter $\b$, we need to solve an unconstrained optimization problem for which the behavior of the CBO method has been broadly analyzed, e.g. in ^^>\cite{carrillo2018analytical,carrillo2019consensus, pinnau2017consensus, fornasier2021consensusbased,fornasier2021convergence, jin2020convergence}. In the constrained settings, though, the optimal value $\bar \b$ is unknown and the system \eqref{eq:sde} may concentrate around an infeasibile points, if $\b$ is not sufficiently large. 

Instead of solving the penalty subproblem \eqref{pb:pen} multiple times, we tune $\b$ during the computation  by adopting the following strategy. 
At each iterate $k$, we update the penalty parameter depending on the  constraint violation of the particles.  In particular, given an (algorithmic) parameter $\theta_k$,  we the tolerance $1/ \sqrt{\theta_k}$ and  evaluate whether the ensemble average on $r$ is below this tolerance: 
\be
\frac 1N \sum_{i=1}^N r(X_k^i)  \leq \frac{1}{\sqrt{\theta_k}}\,.
\label{eq:constr1}
\ee
If the inequality holds true, we  decrease the tolerance by setting $$\theta_{k+1} = \eta_\theta \theta_{k},$$ for some $\eta_\theta> 1$. Otherwise, we increase the penalty parameter by a factor $\eta_\b>1$, 
$$\b_{k+1} = \eta_\b \b_{k},$$
 and increase the tolerance by setting $\theta_{k+1} = \min \{ \theta_k/\eta_{\theta}, \theta_0\}$. The complete adaptive strategy is summarized in Algorithm \ref{alg:iter}.

\begin{algorithm}%[H]
\SetAlgoLined
\DontPrintSemicolon
 Initialize interaction parameters: $\b_0,\, \sigma,\, \lambda$\;
 Initialize algorithm parameters: $\theta_0,\, \eta_\theta, \eta_\b, \Delta t, K$\;  
 Set initial conditions $f_0$\;
 
 $X^i_0 \sim f_0 \quad \forall\, i =1, \dots, N$\;
 \For{k = 1, \dots, K}{
    Compute $X^\a_{k-1}$ according to \eqref{eq:Xa} \;
	 $B^i_{k-1} \sim \mathcal{N}(0,1) \quad \forall i =1, \dots, N$\;
    Compute $X_i^k$ according to \eqref{eq:iter} $\quad \forall\, i =1, \dots, N$\;
    
     \eIf{the feasibility check (\eqref{eq:constr1} or \eqref{eq:constr2}) holds}{
   $\theta_{k+1}  =\eta_\theta\, \theta_k $  \;
   $\b_{k+1}  =\b_k $  \;  
   }{
   $\theta_{k+1}  = \min \{ \theta_k/\eta_{\theta}, \theta_0\} $  \; 
   $\b_{k+1}  =\eta_\b \,\b_k $  \; 
  }
 }
 \caption{Update strategy for penalized CBO} 
 \label{alg:iter}
 
\end{algorithm}
%\clearpage
Concerning the update strategy, some remarks are in order. 
\begin{remark}^^>
	\begin{itemize}
		\item Instead of \eqref{eq:constr1} we may also check the feasibility of the particles using a weighted expectation of $r(x)$:
		\be
		\frac 1{Z_\alpha} \sum_{i=1}^N r(X_k^i) \exp\left(-\alpha \pen(X_k^i,\beta)\right)  \leq \frac{1}{\sqrt{\theta_k}}\,,
		\label{eq:constr2}
		\ee
		as present in the definition of $X_k^\alpha$, see \eqref{eq:Xa}. 
		This condition will be use also used in the numerical examples and we discuss therein its performance. The theoretical analysis following is based on \eqref{eq:constr1}. 
		
		\item While the penalty parameter $\b_k$ is an increasing sequence, the tolerance may both decrease and, in particular, increase up to its initial value. This 
		is not required for the convergence analysis. Numerically, we have observed that this  prevents  the feasibility  to became too restrictive as the particles evolves.
		
	\end{itemize}

\end{remark}

To illustrate the above algorithm let us consider the one--dimensional problem with box constraints 
\be
\min_{x \in \RR} j(x) :=\frac{x^4}5 - 2x^2 + x + 10 \quad \text{subject to} \quad x \geq -1.5 \,.
\label{pb:test1}
\ee
The polynomial objective function attains its global minimum in $\hat x = - 2.3519$, while the solution of the constrained problem is  $\eqref{pb:test1}$ is $x^* = -1.5$. By adding the $\ell_1$-penalization term $r(x) = \|\max(0,-x-1.5)\|_1$, $P(x, \b)$ is exact 
for $\b \geq \bar \b = 4.3$. 

Figure \ref{fig:1} shows the evolution of the particle system at different iteration steps (here denoted as time). Starting form a normal distribution of the initial ensemble of particles and $\b_0=0.1$,  they converge towards the objective function infeasible minimum (Fig.\ref{fig:1a}). As the parameter $\beta_k$ increases  as $k$ increases, the particles are moving away from $\hat x$ (Fig.\ref{fig:1b}) and, once $\beta_k \geq \bar \b$, consensus at the solution $x^*$ is reached, see Fig. \ref{fig:1c}. 

\begin{figure}[h]
\begin{subfigure}{0.33\linewidth}
  \centering
  \includegraphics[trim= 7.5cm 10.5cm 7.5cm 10.5cm, clip, width=1\linewidth]{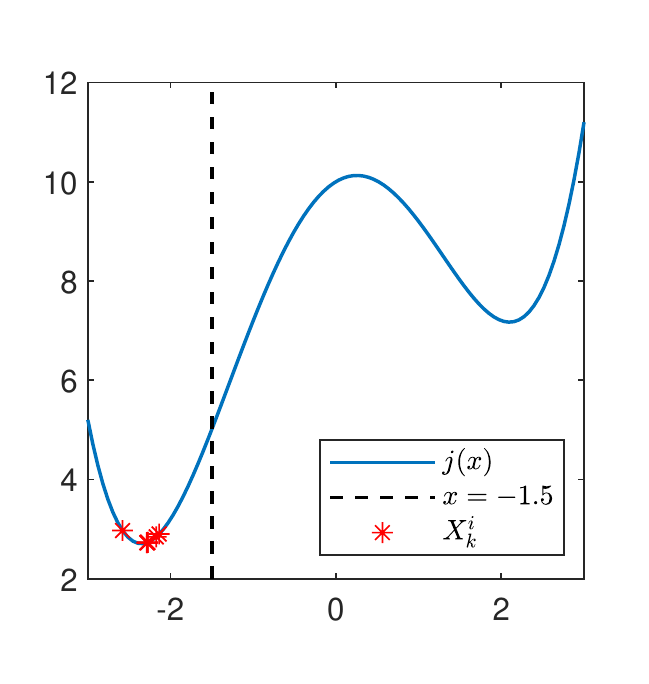}
  \caption{$t = 0.5,\, \beta/\bar\beta \simeq 0.19 $.}
 \label{fig:1a}
\end{subfigure}%
\begin{subfigure}{0.33\linewidth}
  \centering
  \includegraphics[trim= 7.5cm 10.5cm 7.5cm 10.5cm, clip, width=1\linewidth]{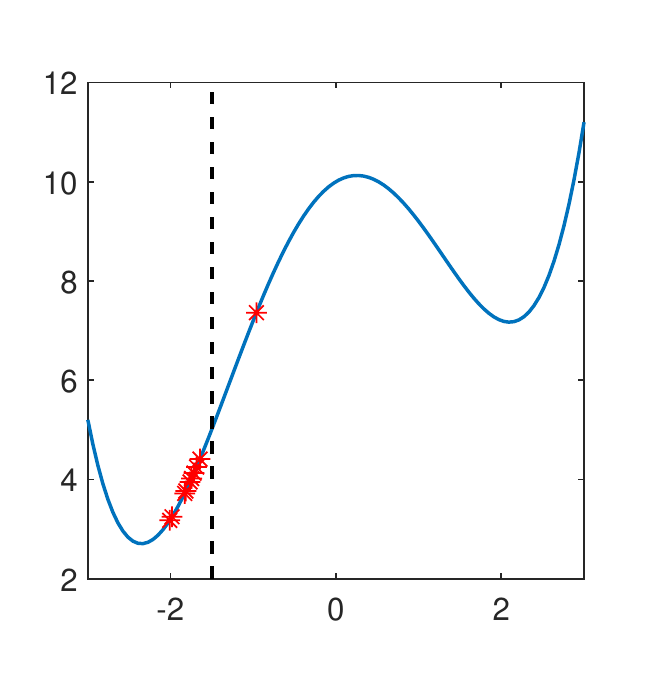}
  \caption{$t = 1, \,\beta/\bar\beta \simeq 0.87 $.}
 \label{fig:1b}
\end{subfigure}%
\begin{subfigure}{0.33\linewidth}
  \centering
  \includegraphics[trim= 7.5cm 10.5cm 7.5cm 10.5cm, clip, width=1\linewidth]{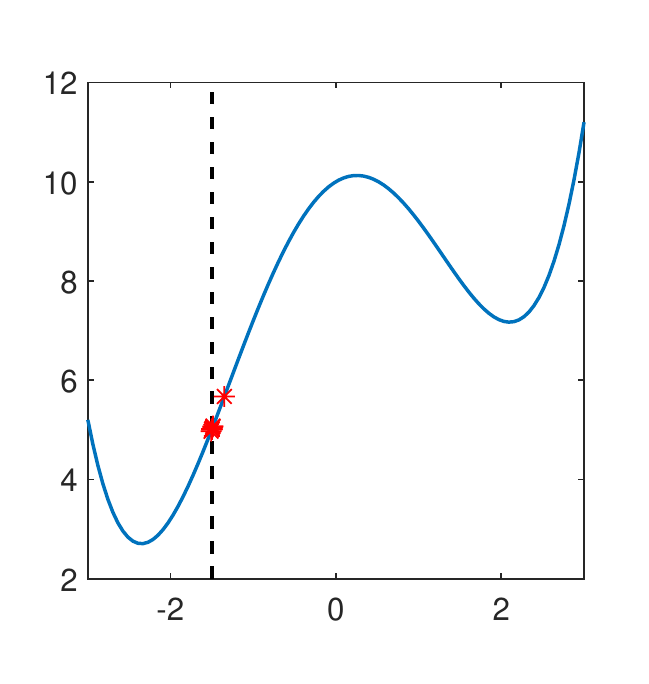}
  \caption{$t = 1.5,\, \beta/\bar\beta \simeq 1.05 $.}
 \label{fig:1c}
\end{subfigure}%

\caption{Algorithm \ref{alg:iter} applied to problem \eqref{pb:test1} with $N=10$ particles. The plots represent the particle distribution at different times $t_k = k\Delta t, \Delta t = 10^{-2}
$. The parameters are $\lambda = 1, \sigma = 10, \b_0 =0.1, \theta_0 = 1, \eta_\b = 1.1, \eta_\theta = 1.1, \Delta t = 10^{-2}.$}
\label{fig:1}
\end{figure}

\begin{figure}[ht]
	\centering
	 \includegraphics[trim= 6cm 10cm 6cm 10cm, clip, width=.475\linewidth]{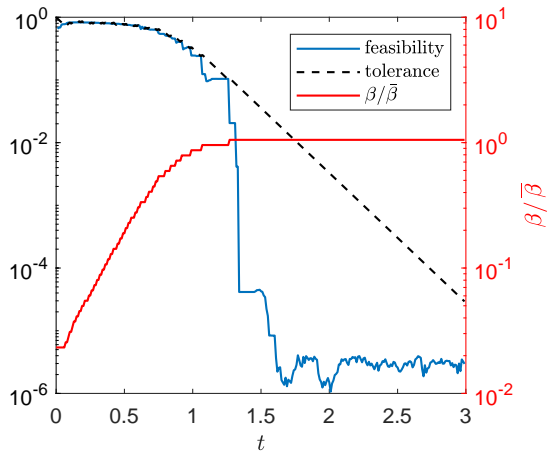}
	\caption{The plot compares the evolution of the feasibility violation (LHS of \eqref{eq:constr2}) with the tolerance $1/\sqrt{\theta}$ and the penalty parameter. We note that $\b$ is nearly constant once the value $\bar \b$ is reached.}
	\label{fig:1pen}
\end{figure}

The feasibility violation (calculated according to \eqref{eq:constr2}) is compared to the tolerance in Figure \ref{fig:1pen}. We note that, up until time $t = 1.3$, the feasibility condition is violated at almost every time, but satisfied at subsequent times. At time or iteration $t=1.3$ the penalty parameter $\b$ is close to $\bar\b$ and hence, the penalty subproblem is  equivalent to the constrained optimization problem.

The test shows  two distinct stages of the algorithm, that will be also discussed analytically: a first where $\b_k$ is smaller than (the general unknown value) $\bar \b$ and a second stage where $\b_k \geq \bar \b$. Clearly, if the adaptive strategy is not capable to reach the second stage, the optimization method will not succeed.  We will analytically investigate this possibility in the next section.

%\clearpage 

\section{Convergence analysis}
\label{sec:3}

In order to analyze the convergence properties of the proposed methods, we introduce the so-called mean--field approximation of Algorithm \ref{alg:iter} which is continuous-in-time and, informally,  exact when the number of particles is infinite. So, we present the main convergence result, Theorem \ref{t:main}. The proof will be provided by first studying the algorithm behavior when, during the computation, the penalty subproblem \eqref{pb:pen} is equivalent to the original constrained problem \eqref{pb}. Secondly, we look at the case where they are not equivalent, i.e. the penalty parameter $\b$ is not sufficiently large and needs to be updated.

\subsection{Mean-field approximation and main result}
\label{sec:31}

Given an ensemble of $N \in \mathbb{N}$  stochastic processes $\{(X_t^i)_{i = 1}^N\;|\; t>0\}$ which take values in $\RR^d$ and such that $t \mapsto X_t^i$ is continuous almost everywhere, we denote $f^N(t) = \sum_{i=1}^N \delta_{X_t^i}$ the random empirical measure at time $t$. We also denote by $\mathcal{P}(\RR^d)$ the space of Borel probability measures and with finite $q$-th moment, i.e., $\mathcal{P}_q(\RR^d) = \{ f \in \mathcal{P}(\RR^d) \,:\,  \int \|x\|^q \d f(x) < \infty \}$.

Following^^>\cite{pinnau2017consensus,carrillo2018analytical}, we consider the Euler--Maruyama  time discretization 
 \eqref{eq:iter}  of the following system of $N$ stochastic differential equations  and assume  isotropic diffusion \eqref{eq:iso}: 

\be
\begin{split}
dX_t^i &= -\lambda \left(X_t^i - x_\a(f^N)\right)dt + \sigma \| X_t^i - x_\a(f^N)\|dB_t^i \\
\lim_{t\to 0 } X_t^i &= X_0^i \qquad X_0^i \sim f_0. 
\label{eq:sde}
\end{split}
\ee
Here,  the local estimate of the global minimum of $P(\cdot,\b)$ at the time $t\geq0$ is given by  $x_\a(f^N)$ and defined by 
\begin{equation} \notag
x_\a(f^N) = \frac1 Z_\alpha \int x\, e^{-\a \pen(x,\b)} \d f^N(t,x)\,. 
\end{equation}

The well--posedness of system has been investigated e.g. in^^>\cite{carrillo2018analytical}. We recall, if $f_0 \in \mathcal{P}_4(\RR^d)$ and $\pen$ is locally Lipschitz, then the system \eqref{eq:sde} admits a unique strong solution $\{ (X_t^i)_{i=1}^N \;|\; t\geq 0\}$ with empirical distribution $f^N(t) \in \mathcal{P}({\RR^d})$.

Equation \eqref{eq:sde} is the continuous-in-time limit of the iterative method \eqref{eq:iter}. The  continuous-in-time formulation allows to derive a 
mean-field limit in the large particle limit  $N\gg1$. To be more specific, under mild assumptions on the objective function $P$, if $\{ (X_t^i)_{i=1}^N \;|\; t \in [0,T]\}$ is a solution to \eqref{eq:sde} for some time horizon $T>0$, and $f_0 \in \mathcal{P}_4(\RR^d)$, then for all $t \in [0,T]$
\[f^N(t) \; \longrightarrow\, f(t) \quad \text{in law as}\; N \to \infty.  \]
Moreover, the limit $f \in  \mathcal{C}([0,T], \mathcal{P}_4( \RR^d))$ is the unique weak solution to the (deterministic) Fokker-Planck equation
\be
\partial_t f = \lambda\, \text{div}((x - x_\a(f))f) + \frac{\sigma^2}2 \Delta (\|x - x_\a(f)\|^2f)\,,
\label{eq:pde}
\ee
with initial data $f_0$ over the function space
\be
\notag
\mathcal{C}_*^2(\RR^d) := \{ \phi \in \mathcal{C}^2(\RR^d)\; | \; \| \nabla \phi (x) \| \leq C( 1+ \|x\|) \;\; \text{and} \;\; \sup_{x  \in \RR^d} |\Delta \phi(x) | < \infty \}\,.
\ee
That is, $f(t)$ fulfills for all $\phi \in \mathcal{C}_*^2(\RR^d)$ and for all $t\in (0,T)$
\be
\frac{d}{dt} \int \phi(x) \d f(t,x) = -\lambda \int (x- x_\a(f))\cdot \nabla \phi(x) \d f(t,x) 
+ \frac{\sigma^2}2 \int \| x - x_\a(f)\|^2 \Delta \phi(x) \d f(t,x)
\label{eq:ws}
\ee
and the initial data is attained pointwise, i.e., $\lim_{t \to 0} f(t) = f_0$. 
We refer to^^>\cite{huang2021meanfield} for a detailed discussion and proof of the mean-field limit and to ^^>\cite{carrillo2018analytical,fornasier2021consensusbased} for the well-posedness of \eqref{eq:pde}. 
We recall, that \eqref{eq:pde} strongly depends on the function $P$ through the term $x_\a(f)$, which is nonlinear and nonlocal.

Since the mean-field model is more amenable to theoretical analysis, we introduce the mean-field counterpart of Algorithm \ref{alg:iter}, Algorithm \ref{alg:mf}. In the following  we  study the large time behavior of the particle distribution $f$ constructed through Algorithm \ref{alg:mf}.

\begin{algorithm}[h]
	\SetAlgoLined
	\DontPrintSemicolon
	%\NoCaptionOfAlgo
	Initialize interaction parameters: $\b_0,\, \sigma,\, \lambda$ with $2\lambda > d \cdot \sigma^2$\;
	Initialize algorithm parameters: $\theta_0,\, \eta_\theta, \eta_\b, \Delta t, K$\;  
	Set initial conditions $f_0$\;
	
	\For{k = 1, \dots, K}{
		Set $g^k$ to be the solution to \eqref{eq:pde} over the time interval $[t_{k-1},t_k]$
		with penalty parameter $\b_k$ and initial data $f_{k-1}$\;    
		$f(t) = g^k(t)$ for all $t \in [t_{k-1},t_k]$\;
		
		$f_k = f(t_k)$\;
		
		\eIf{the feasibility check \eqref{eq:conmf} holds}{
			$\theta_{k+1}  =\eta_\theta\, \theta_k $  \;
			$\b_{k+1}  =\b_k $  \;  
		}{
			$\theta_{k+1}  = \min \{ \theta_k/\eta_{\theta}, \theta_0\} $  \; 
			$\b_{k+1}  =\eta_\b \,\b_k $  \; 
		}
	}
	\caption{Mean-field penalized CBO} 
	\label{alg:mf}
	
\end{algorithm}

In Algorithm \ref{alg:mf}, we evolve the the solution of the Fokker-Planck equation \eqref{eq:pde} for a (possibly short) time interval $\Delta t$  and at the discrete times $t_k = k \Delta t, 1\leq k \leq K$, we consider the  feasibility of the solution. 
Written in terms of the mean-field distribution $f$, the feasibility condition \eqref{eq:constr1} reads then 
\be
\int r(x)\, \d f_k(x) \leq \frac{1}{\sqrt{\theta_k}}\,
\label{eq:conmf}
\ee
where $f_k=f(t_k)$ is the particle distribution at time $t_k$. 

At the end of the computation, we have that $f \in \mathcal{C}([0, K\Delta t], \mathcal{P}_4(\RR^d))$ and, when restricted to the time interval $[t_{k-1}, t_k]$ for some $1\leq k \leq K$, it solves \eqref{eq:pde} with initial data $f_{k-1} = f(t_{k-1})$ and penalty parameter $\beta_k$.

For the analysis in the following paragraph we assume 
\begin{enumerate}
\item[(A1.3)]
$r(x) := \dist(x, \MM) = \min_{y\in \MM} \|x - y\|\,.$
\end{enumerate}

To analytically study the convergence of  $f$ to the solution $x^*$ of \eqref{pb} we introduce the functional
\be
\mathcal{V}(t) := \frac 12 \int \| x- x^*\|^2 \d f(t,x)\,.
\label{eq:V}
\ee
which gives a measure on how far $f(t)$ is from $x^*$. 

For all $\b = \b_k, \,k \geq 1$,
we also assume
$P(x,\beta)$ to be $p$-conditioned in a neighborhood of the minimizer $x_\beta^*:= \amin_{x \in \RR^d} P(x,\beta)$, for some $p\geq2$. Namely, we assume that 
\begin{itemize}
\item[(A2)] $ P(x,\beta) - \inf_{x \in \RR^d} P(x,\beta) \geq C'_1\| x - x^*\|^p \quad \text{for all}\; \|x\|^2 \leq R'_1$
\end{itemize}
and that, outside this neighborhood of $x^*$, $P(x,\beta)$ is not arbitrary close to $\inf_{x \in \RR^d} P(x,\beta)$:
\begin{itemize}
\item[(A3)]$P(x,\beta) - \inf_{x \in \RR^d}P(x,\beta) \geq P_\beta^\infty \quad \text{for all}\; \|x\|^2 \geq R'_1\,$.
\end{itemize}
Introduced in^^>\cite{vainberg1970, zolezzi1978}, the notion of conditioning is a common tool on optimization literature, see for example^^>\cite{garrigos2020convergence}, and it is also known as \textit{growth condition}^^>\cite{penot1996}.

Finally, we present the main convergence result.
\begin{theorem} Assume (A1)--(A3) for all $P(x, \beta_k)\,\,k\geq 1$. Let $f$ be constructed according to line 6 of Algorithm \ref{alg:mf} with initial datum $f_0 \in \mathcal{P}_4( \RR^d)$ such that
\be
\{x_{\b_k}^* \in \RR^d\, |\;\; \text{for all}\;\;\b_k < \bar \b  \} \cup \{x^*\} \subseteq \textup{supp}(f_0)\,,
\label{eq:tcond}
\ee
where $x_\beta^*= \amin_{x \in \RR^d} P(x,\beta)$.

Given an accuracy $\ve \in (0,\mathcal{V}(0))$, if $K, \theta_0$ and $\a$ are large enough, then there exists $T^* \in [0, K\Delta t]$ such that
\[
\mathcal{V}(T^*) = \ve\,.
\]
\label{t:main}
\end{theorem} 
We prove the theorem by studying the two different situations separately, first when $\b_k \geq \bar \beta $  at some step $k$ and then when $\b_k < \bar \beta$. The collected results will be then connected to provide a proof of \cref{t:main} at the end of this section.

Some remarks are in order.
\begin{remark}^^> 
\begin{itemize}
\item
Note that, assuming (A1.3) and \eqref{eq:conmf} allows for the following interpretation in Algorithm \ref{alg:mf}: the feasibility of a given ensemble is measured as the  expectation of $r(x)= \dist(x,\MM)$ on this set. 
\item By the update rule of $\b_k$ in Algorithm \ref{alg:mf}, the set $\{x_{\b_k}^* \in \RR^d\, |\;\; \text{for all}\;\;\b_k < \bar \b  \}$ contains a finite number of points, which all belong to $\mathcal{M}^c$ due to (A1.2).
\end{itemize}
\end{remark}

\subsection{Convergence estimates of the penalty subproblem}

When the penalty parameter $\b$ is fixed, the CBO particle method is capable of successfully solving the penalty subproblem \eqref{pb:pen} provided that $P$ satisfies (A2),(A3) ^^>\cite{fornasier2021consensusbased}. In particular, we recall the following result.

\begin{theorem}{{\em^^>\cite[Theorem 12]{fornasier2021consensusbased}}}
Let $\b>0$ be fixed, $P(x, \b) \in \mathcal{C}(\RR^d)$ satisfy assumptions (A2), (A3) and $\lambda, \sigma$ be such that $2\lambda > d \cdot \sigma^2$. Moreover, let $f_0 \in\mathcal{P}_4(\RR^d)$ satisfy
\be
\notag
x_\beta^* \in \textup{supp}(f_0)
\ee 
where $x_\beta^* =  \amin_{x \in \RR^d} P(x, \b)$.

For any accuracy $\ve \in (0, \mathcal{V}(0))$, there exits a time $T^*$ and $\a_0$ large enough such that for all $\a>\a_0$ we have $\mathcal{V}_\beta(T^*) = \ve$, if $f \in\mathcal{C}([0,T]), \mathcal{P}_4(\RR^d))$ is a weak solution to the Fokker-Planck equation \eqref{eq:pde} on the time interval $[0,T], T>T^*,$ with initial conditions $f_0$, 

Furthermore, for all $t \in [0,T^*]$, we have exponential decay
\be
\mathcal{V}_\beta(t) \leq \mathcal{V}_\beta(0) \exp\left ( -\left(\lambda - \frac{d\sigma^2}{2}\right)t \right)\,,
\label{eq:est1}
\ee
\label{t:conv}
where, similar to \eqref{eq:V}, we denote
\be
\notag
\mathcal{V_\beta}(t) := \frac 12 \int \| x- x_\beta^*\|^2 \d f(t,x)\,.
\ee

\end{theorem}

Considering  the proof in^^>\cite{fornasier2021consensusbased}, we also obtain  the estimate
\be
\| x_\beta^* - x_\a(f)\| \leq C \sqrt{\mathcal{V_\beta}(t)}
\label{eq:est2}
\ee
and  
\be
x_\beta^* \in \textup{supp}\left(f(t)\right) \quad \text{for all} \quad t \in [0,T^*]\,.
\label{eq:supp}
\ee

Clearly, in Algorithm \ref{alg:mf} the function $P(x, \beta_k)$ varies due to the adaptive strategy and we cannot directly apply the above convergence result. Nevertheless, it will constitutes an essential tool for the following analysis.

\begin{remark} Recently, the case of  anisotropic diffusion introduced in^^>\cite{carrillo2019consensus} has been analyzed in^^>\cite{fornasier2021convergence} by proving rigorous convergence in the mean-field limit. In this case, there is no dependence of 
	the decay rate \eqref{eq:est1} on the dimension $d$. Following the same arguments, we expect that it is also possible to extend the previous theorem to the anisotropic case. We leave the details of this extension to further research. Numerically, however we will explore this extension in several test cases.
\end{remark}

\subsection{Convergence for exact penalization $\b\geq \bar \b$}

We start by analyzing the case where, at a certain (time) step $k_0$, we have $\beta_{k_0} \geq \bar \b$ and so the minimum of the penalty subproblem is equivalent to the minimum of the constrained problem according to $(A1.1)$. Note that typically $\bar \b$ is not known.

 In these settings, $\b_k \geq \bar \b$ for all $k \geq k_0$ and in particular $x^*_{\beta_k} = x^*$ is independent of $\b_k$ and is the solution to \eqref{pb}. 

\begin{proposition} 
Let $k_0$ be at time step such that $\b_{k_0}\geq \bar \b$. Assume (A1)--(A3) for $P(x, \beta_{k_0})$ and $f_{k_0}$ to be such that
\be \notag
x^* \in \textup{supp}(f_{k_0})\,.
\ee
For a given accuracy $\ve \in (0, \mathcal{V}(t_{k_0}))$, $\mathcal{V}(T^*) =  \ve$ for some time $T^* \in [0, K \Delta t] $, if $\a$ and $K$ are large enough.
\label{p:1}
\end{proposition}
\begin{proof}
By contradiction, let as assume that such a $T^*>0$ does not exists. Hence, the desired accuracy is not reached in all intervals $[t_k, t_{k+1}]$ for $k> k_0$. We note that, since $P(x, \beta)$ is increasing on $\beta$, (A2) and (A3) hold for all $P(x, \beta_k)$ with $\beta_k \geq \beta_{k_0}$. Also, from \eqref{eq:supp}, $x^* \in \textup{supp}(f_k)$ for all $k> k_0$, which means that we can iteratively apply Theorem \ref{t:conv} to obtain the decay estimate
\be
\label{eq:decay}
\mathcal{V}(t_k) \leq \mathcal{V}(t_{k-1})  \exp\left ( -\left(\lambda - \frac{d\sigma^2}{2}\right)\Delta t \right) \leq  \mathcal{V}(t_{k_0})  \exp\left ( -\left(\lambda - \frac{d\sigma^2}{2}\right)(k- k_0)\Delta t \right)
\ee
for all $k>k_0$, provided that $\a$ is large enough. Therefore, for $k$ sufficiently large we obtain a contradiction.
\end{proof}

If $\b_k\geq\bar \b$, then due to (A1.1) $x^*$ is feasible. Under assumption (A1.3), the penalty term $r(x)$ can be bounded by $\| x - x^*\|$. As a consequence, we obtain that the feasibility condition \eqref{eq:conmf} is satisfied until the desired accuracy is reached. However, this happens only if the particle distribution is sufficiently close to the minimizer.

\begin{proposition} Let the assumptions of Proposition \ref{p:1} hold.  If $\eta_{\theta}\leq \exp\left((\lambda - \frac{d\sigma^2}{2})\Delta t \right)$ and
\be
\mathcal{V}(t_{k_0}) \leq \frac{1}{ 2\theta_{k_0}}\,,
\label{eq:p32}
\ee
then the feasibility condition is satisfied for all $t_k \in [t_{k_0}, T^*]$.
\label{p32}
\end{proposition}
\begin{proof}
Since $x^* \in \mathcal{M}$ and due to (A1.3)  $r(x) = \text{dist}(x, \mathcal{M}) \leq \|x - x^*\|$. Together with the Jensen's inequality, this yields to 
\be
\notag
\int r(x) \d f_k (x) \leq \int \| x- x^*\| \d f_k (x) \leq \left( \int \| x - x^*\|^2 \d f_k(x) \right)^{\frac12} = \sqrt{ 2 \mathcal{V}(t_k)}\,.
\ee 
By the choice of $\eta_\theta$,
\begin{equation*}
e^{-\left (\lambda - \frac{d\sigma^2}{2}\right )(k - k_0) \Delta t} \leq \eta^{-(k - k_0)}\quad \forall \; k \geq k_0 \,.
\end{equation*}

Using this bound and the decay rate \eqref{eq:decay} in the inequality above, we obtain
\begin{align*}
\int r(x) \d f_k (x)  \leq  \left( 2 \mathcal{V}(t_{k_0}) e^{-\left (\lambda - \frac{d\sigma^2}{2}\right )(k - k_0) \Delta t}\right)^{\frac12}  \leq  \left( \frac{2 \mathcal{V}(t_{k_0})} {\eta^{k - k_0}} \right)^{\frac12}
\leq \frac{1}{\sqrt{\theta_{k_0}\eta^{k - k_0}}} \,,
\end{align*}
where we used assumption \eqref{eq:p32}. The latter allows to show that the particle distribution $f_{k_0}$ is concentrated close to the minimizer. We conclude by noting that $\theta_k \leq \theta_{k_0} \eta_{\theta}^{k-k_0}$ for all $k \geq k_0$.
\end{proof}

In the remaining part of the analysis, we will consider the case where at some $ k$, $\b_k$ is smaller than $\bar \b$. In particular, we show that then the feasibility condition will be necessarily violated at some subsequent time step. If this happens, the penalty parameter will be updated and eventually go beyond the threshold value $\bar \b$. Due to (A1.1)  and  Proposition \ref{p:1}, $f$ then concentrates close to the solution $x^*$ of the constrained  problem \eqref{pb}.

\subsection{Stability estimate of the constraint violation} 

First, we collect some stability estimates on the constraint violation \eqref{eq:conmf}.

As the penalty term $r(x)$ is not differentiable, we  approximate $r(x)$ by a function that belongs to the space $\mathcal{C}_*^2(\RR^d)$. In this way, the large time behavior of $\int r \,\d f(t)$ using the definition of weak solution \eqref{eq:ws} is estimated.

We perform approximation by mollification by means of the well-studied mollifier $\phi \in \mathcal{C}^{\infty}(\RR^d)$

\begin{equation*}
\phi(x) := 
	\begin{cases}
	C_{\phi} \exp \left( - \frac1{\|x\|^2 - 1} \right), & \text{if} \; \|x\| <1 \\
	0, & \text{if}\;\|x\| \geq 1\,,
	\end{cases}
\end{equation*}
where $C_{\phi}$ is chosen such that $\int_{\RR^d} \phi(x) dx = 1$. For a certain $\e>0$ we set $\phi_\e$ to be
\begin{equation*}
\phi_\e(x) = \frac1 {\e^d} \phi \left(\frac x \e\right)\quad \text{and define} \quad  r^{\e} := \phi_\e \ast r\,,
\end{equation*}
where $\ast$ denotes the convolution on $\RR^d$. 

For a given $f \in \mathcal{P}_2(\RR^d)$, it holds 
\be 
\left | \int r(x) \d f(x) -  \int r^\e(x) \d f(x) \right | \leq \e. 
\ee
and we have  
$r^\e \in C_*^2(\RR^d)$. 

\begin{lemma}\label{l:mol}
For $\e>0$, the following estimates hold:
\begin{align*} 
\| \nabla r^\e(x) \| \leq 1 \quad \text{and} \quad 
| \Delta r^\e(x) | \leq \frac{d C}{\e^2} 
\quad \text{for all}\quad  x \in \RR^d
\,,
\end{align*}
from which follows $r^{\e} \in C_*^2(\RR^d)$.
\end{lemma}
\begin{proof}
Even though $r(x)$ is not differentiable, for a given direction $v \in \RR^d$, the directional derivative $Dr(x, v)$ exists almost everywhere. This will allow us to derive an explicit definition of $\nabla r^\e(x)\cdot v$ in terms of $Dr(x,v)$.

Let $x$ belong to the interior of $\MM$, $x \in \MM^o$. Clearly, $Dr(x, v) = 0$ for all $v \in \RR^d$. For $x \in \MM^c$, we note that
$g_y(x) := \| x- y\|$ is $C^2$ in a neighborhood of $x$ for $y \in \MM$ and that the penalty term can be written as $r(x) = \min_{y \in \MM}g_y(x)$. Under these settings, there exists an explicit representation of the directional derivative, cf.^^>\cite[Theorem 1, pp. 22]{Danskin1967TheTO}.
Indeed, let $Y(x)$ be the set of those $y \in \MM$ which yield the minimum to $\| x- y\|$ for $x$ fixed, $Y(x) = \{ y \in \mathcal{M}\,|\, r(x) = \| x- y\|\}$. For all $v \in \RR^d$, the directional derivative is then given by

\be \notag
Dr (x, v) = \min_{y \in Y(x)}\nabla g_y(x)  \cdot v = \min_{y \in Y(x)}\frac{(x - y)}{\|x -y\|}\cdot v\,.
\ee 

We recall $\partial \MM$ has Lebesgue measure zero, and hence $Dr(x, v)$ is defined almost everywhere. Additionally, $Dr(x, v)$ is bounded a.e.: $|Dr(x,v)| \leq \| v \|$.
 Under these assumptions, the directional derivative of $r^\e(x)$ is expressed as a mollification of the directional derivatives of $r(x)$:

\be
\nabla r^\e(x) \cdot v = \int \phi_\e(z) Dr(x-z, v) \,\d z = \int \phi_\e(z) \min_{y \in Y(x-z)}\frac{(x -z - y)}{\|x-z -y\|}\cdot v \,\d z \,.
\label{eq:direct}
\ee

We refer to^^>\cite[Lemma 3.10]{jongen2008a} for a proof of \eqref{eq:direct}. As a direct consequence, we obtain the bound $\| \nabla r^\e(x) \| \leq 1$.

In order to bound the second order derivatives of $r^\e(x)$, let us consider the gradient of $\nabla r^\e(x)\cdot v$, for a fixed $v \in \RR^d$. By the dominated convergence theorem, we rewrite the gradient as

\begin{align*}
 \nabla^2 r^\e(x)\cdot v &= \nabla_x \int \phi_\e(z) D r(x-z, v)\, \d z  = \nabla_x \int \phi_\e(x-z) D r(z, v)\, \d z  \\
&= \int \nabla_x(\phi_\e(x-z)) D r(z, v) \, \d z 
\end{align*}
and obtain
\be \notag
  \| \nabla^2 r^\e(x)\cdot v \| \leq  \int \|\nabla_x(\phi_\e(x-z))\| |D r(z, v)|\, \d z 
 \leq \|v\|\,  \int \|\nabla_x(\phi_\e(x-z))\| \,\d z \,.
\ee

We recall that $\nabla_y \phi(y) =  - 2 C_{\phi} \phi(y) y (\|y\|^2 -1)^{-2}$ and it  follows
\begin{align*}
\nabla_y \phi_\e(y) = - \frac{2C_\phi}{\e^d} \phi(y/\e)\frac{y/\e}{(\|y/\e\|^2 -1)^2} \frac 1\e\,.
\end{align*}
Therefore,
\begin{align*}
 \int \|\nabla_y(\phi_\e(y))\| \,\d y & \leq \int_{B(0,\e)} \frac{2C_\phi}{\e^{d+1}} \phi(y/\e)\frac{\|y/\e\|}{(\|y/\e\|^2 -1)^2}\,  \d y \\
 & = \int_{B(0,1)}  \frac{2C_\phi}{\e^{d+1}} \phi(y')\frac{\|y'\|}{(\|y'\|^2 -1)^2} \e^{d-1}\, \d y'  \;\leq\; \frac{C}{\e^2} \;\;
\end{align*}
for some positive constant $C>0$, where we used that $\phi(y')\, \|y'\|\, (\|y'\|^2 -1)^{-2}$ is bounded.

In particular, we have $\| \nabla^2 r^\e(x) \|_2 \leq C/\e^2$ and  $|\Delta r^\e| \leq d C/ \e^2$.
\end{proof}

Given a certain time interval $[ t_k, t_{k+1}]$, we are now able to bound the constraint violation $\int r \, \d f(t^*)$ for $t \in (t_k, t_{k+1})$.

\begin{lemma}
Let $ t^* \in (t_k, t_{k+1})$ and let the estimates \eqref{eq:est1},\eqref{eq:est2} of  Theorem \ref{t:conv} hold true for all $t \in (t_k, t^*)$.
Then 
\be
\left |  \int r^\e(x) \, \d f(t_k,x) - \int r^\e(x) \, \d f(t^*,x) \right |  \leq \left (\lambda + \frac{ d \sigma}{2 \e^2}C_1 \right) C_2 \max\left \{\mathcal{V}_\beta(t_k), \sqrt{\mathcal{V}_\beta(t_k)}\right\}  \, \Delta t
\ee
for some constants $C_1, C_2 >0$.
\label{l:1}
\end{lemma} 
\begin{proof}
We start by applying the definition of weak solution \eqref{eq:ws} and the Cauchy-Schwartz inequality
\begin{align*}
\frac{d}{dt} \int r^\e (x) \,\d f(t,x) &= - \lambda \int \nabla r^\e(x)\cdot (x-x_\a) \, \d f(t,x) + \frac{\sigma^2}{2} \int \Delta r^\e (x) \| x - x_\a\|^2\,\d f(t,x) \\
& \leq \lambda \int \| x- x_\a\| \d f(t,x) + \frac{d\sigma^2}{2 \e^2} C_1 \int \| x- x_\a\|^2\, \d f(t,x)\,.
\end{align*}
By  estimate \eqref{eq:est2},
\begin{align}
\int \| x - x_\a \|^2 \, \d f(t,x) &\leq \int \|x-x_\beta^*\|^2\, \d f(t,x)  +  \|x_\beta^* - x_\a \|^2 \notag \\
&\leq C' \mathcal{V}_\beta(t) + 2\mathcal{V}_\beta(t) \leq (C'+2) \mathcal{V}_\beta(t_k)\, \label{eq:est3}
\end{align}
since $\mathcal{V}_\beta$ is non-increasing in the interval $(t_k, t^*)$, see \eqref{eq:est1}.
Altogether, we obtain for some constant $C_2>0$

\be \notag
 \frac{d}{dt} \int r^\e(x) \, \d f(t,x) \leq \left (\lambda + \frac{ d \sigma}{2 \e^2}C_1 \right) C_2 \max \left \{\mathcal{V}_\beta(t_k), \sqrt{\mathcal{V}_\beta(t_k)}\right \} 
\ee
which yields
\be
\left |  \int r^\e(x) \, \d f(t_k,x) - \int r^\e(x) \, \d f(t^*,x) \right |  \leq \left (\lambda + \frac{ d \sigma}{2 \e^2}C_1 \right) C_2 \max\left \{\mathcal{V}_\beta(t_k), \sqrt{\mathcal{V}_\beta(t_k)}\right\}  \; \Delta t\, . 
\notag
\ee

\end{proof}

\subsection{The case with penalty parameter update}

Before proving Theorem \ref{t:main} we first consider the case when the penalty subproblem is not equivalent to \eqref{pb}, i.e.,  $\b_{k_0} < \bar \b$ at a certain step $k_0$. We show that  $\b_k$ will necessary increase at a subsequent time $t_k$, $k>k_0$. The proof will be done by contradiction.  

Indeed, if the penalty parameter $\b_k=\b_{k_0}$ remains constant for all $t\geq t_{k_0}$, by applying Theorem \ref{t:conv}, we know that the density $f$ will converge to the minimizer $x_0^*$ of $P(x, \b_{k_0})$.

\par 
Since $\b = \b_{k_0} < \bar \b$, $x_0^*$  is not feasible and we expect the feasibility condition to be violated. This is summarized in the following proposition.  

\begin{proposition}
Let $k_0$ be a time step such that $\b_{k_0}< \bar \b$. Assume (A1)--(A3) for $P(x, \beta_{k_0})$ and $f_{k_0}$
 be such that
\be \notag
x_0^* \in \textup{supp}(f_{k_0})
\ee
where $x_0^* = \amin_{x \in \RR^d} P(x, \beta_{k_0})$.

Then, if 
\be r(x_0^*) > 1/\sqrt{\theta_0}\,,
\label{eq:infeas}
\ee the feasibility condition will be violated at some $k>k_0$, provided that $\a$ and $K$ are sufficiently large.
\label{p:2}
\end{proposition}

For notational simplicity, in the following we denote $\mathcal{V}_{\beta_{k_0}}(t)$ with $\mathcal{V}_0(t)$. We first prove the following auxiliary lemma.
\begin{lemma}
	Assume (A1)--(A3) for $P(x, \beta_{k_0})$.  Let $ t^* \in (t_k, t_{k+1})$ and let the estimates \eqref{eq:est1},\eqref{eq:est2} of Theorem \ref{t:conv} hold for all $t \in (t_k, t^*)$.
Then 
\be \notag
\mathcal{V}_0(t_k) \leq \mathcal{V}_0(t^*) e^{C \Delta t}\,,
\ee
for a $C>0$.
\label{l:2}
\end{lemma}
\begin{proof}
By the identity \eqref{eq:ws}, for $t \in (t_k,t^*)$
\begin{align*} \notag
\frac{d}{dt} \mathcal{V}_0(t)  &= - \lambda \int (x-x_0^*) \cdot (x - x_\a) \, \d f(t, x) + \frac{d \sigma^2}{2} \int \| x - x_\a\|^2 \, \d f(t,x) \\
&\geq - \frac \lambda 2 \int \| x - x_0^*\|^2 \, \d f(t, x)  - \left ( \frac \lambda 2 +\frac{d \sigma^2}{2} \right) \int \| x - x_\a\|^2 \, \d f(t,x)\,,
\end{align*}
where we applied Young's inequality.

As before, see \eqref{eq:est3}, we bound the second integral in terms of $\mathcal{V}_0(t)$ and obtain for some constant $C$, 

\be
\frac{d}{dt} \mathcal{V}_0(t) \geq - C \mathcal{V}_0(t)\, .
\ee

We use Grönwall's inequality to obtain the desired estimate:
\be
\mathcal{V}_0(t^*) \geq \mathcal{V}_0(t_k)e^{-C(t^* - t_k)} \quad \Rightarrow \quad  \mathcal{V}_0(t_k) \leq \mathcal{V}_0(t^*) e^{C(t_k-t^*)} \leq \mathcal{V}_0(t^*) e^{C\Delta t}.
\ee
\end{proof}

\begin{proof}[Proof of Proposition \ref{p:2}]
By contradiction, let us assume that $\b_k = \b_{k_0}$ for all $k>k_0$.

Using the same iterative argument as in the proof of Proposition \ref{p:1}, for any given accuracy $\ve>0$, there exists an $\a$ and a time $T^*$ large enough such that $\mathcal{V}_0(T^*) = \ve$. At the time $T^*$, we bound the constraint violation from below
\begin{align}
\int r(x) \, \d f(T^*, x) &\geq  r(x_0^*) - \int \| x- x_0^*\| \, \d f(t, x)  \notag\\
&\geq r(x_0^*) - \sqrt{2\mathcal{V}_0(T^*)} \; \geq\; r(x_0^*) - \sqrt{2\ve}\,, \;\;
\label{eq:lb}
\end{align}
thanks to the triangular inequality $\text{dist}(x,\MM) \geq \dist(x_0^*, \MM) - \|x-x_0^*\|$.

Let now $t_k < T^*$ be such that $k = \argmax_{j} \{ j\,:\, j\Delta t = t_j <T^* \}$.  We  use of the mollification $r^\e$ to obtain an upper bound on the feasibility violation at the time $T^*$. It holds

\begin{align*}
\int r(x) \, \d f(T^*, x) &\leq \e  
 + \left | \int r^\e (x)\, \d f(T^*,x) - \int r^\e (x)\, \d f(t_k,x)\right |
+\e 
+ \int r (x) \d f(t_k,x)\\
&\leq 2\e + \frac 1{\sqrt{\theta_0}}  + \left | \int r^\e (x)\, \d f(T^*,x) - \int r^\e (x)\, \d f(t_k,x)\right | \,.
\end{align*}

By Lemma \eqref{l:1}, 
\begin{align*}
\left | \int r^\e (x)\, \d f(T^*,x) - \int r^\e (x)\, \d f(t_k,x)\right | &\leq
\left (\lambda + \frac{ d \sigma}{2 \e^2}C_1 \right) C_2 \max\left\{\mathcal{V}(t_k), \sqrt{\mathcal{V}(t_k)}\right\}  \; \Delta t \\
& \leq  \sqrt{\ve}\, 2 \left(\lambda + \frac{ d \sigma}{2 \e^2}C_1 \right) C_2 e^{\frac{C \Delta t}2}\Delta t\,,
\end{align*}
where we have used the estimate and constant $C$ of Lemma \ref{l:2}
\[ \mathcal{V}_0 (t_k) \leq \mathcal{V}_0(T^*) e^{C \Delta t} = \ve e^{C \Delta t}\,,\]
which is smaller than one, if $\ve$ is sufficiently small. Therefore, 
\be
\int r(x) \, \d f(T^*, x)  \leq  2\e + \frac 1{\sqrt{\theta_0}} + \sqrt{\ve} \, 2 \left(\lambda + \frac{ d \sigma}{2 \e^2}C_1 \right) C_2 e^{\frac{C \Delta t}2}\Delta t \notag
\ee
 and, if we choose $\e = \ve^{\frac18}$, we obtain
\be
\int r(x) \, \d f(T^*, x)  \leq \frac 1{\sqrt{\theta_0}} + o(\ve^{\frac14})\,.
\label{eq:ub}
\ee 

Putting together the previous bounds \eqref{eq:lb} and \eqref{eq:ub}, we have
\be  r(x_0^*) - \sqrt{2\ve} \leq \int r(x) \, \d f(T^*, x) \leq \frac 1{\sqrt{\theta_0}} + o(\ve^{\frac14})\,,
\notag
\ee
which yields to 
\be
r(x_0^*) \leq \frac 1{\sqrt{\theta_0}} + o(\ve^{\frac14})\, .
\label{eq:rest}
\ee

For sufficiently small $\ve$, \eqref{eq:rest} contradicts assumption \eqref{eq:infeas}.

\end{proof}

\begin{proof}[Proof of Theorem \ref{t:main}]

If $\b_0$ is such that $\b_0 \geq \bar \b$, since $x^*\in \textup{supp}(f_0)$ we can directly apply \cref{p:1} to conclude that there exists $K, \alpha$ large enough such that $\mathcal{V}(T^*) = \ve$ for some $T^* \in [0, K \Delta t]$.

If this is not the case, we recall from ^^>\cite[Proposition 19, Remark 20]{fornasier2021consensusbased} that, if $f$ is the solution of the Fokker--Planck equation \eqref{eq:pde} over $[0, \Delta t]$, with initial datum $f_0$ it holds
\be
\{x_{\b_k}^* \in \RR^d\, |\;\; \text{for all}\;\;\b_k < \bar \b  \} \cup \{x^*\} \subset \textup{supp}\left(f(t)\right) \quad \forall \; t \in [0, \Delta t]\,.
\ee
By iterating the above inclusion for all the time intervals $[t_{k-1}, t_k]$, we have that
\[\{x_{\b_k}^* \in \RR^d\, |\;\; \text{for all}\;\;\b_k < \bar \b  \} \cup \{x^*\} \subseteq \textup{supp}(f_k) \quad \text{for all}\quad 1\leq k \leq K \,,\]
from which follows in particular that
\[ x_{\beta_k}^* \in \textup{supp}(f_k) \quad \text{for all} \quad \b_k <\bar\b.\]

Let now $\theta_0$ be such that 
\[ r(x_{\b_k}^*) > \frac1{\sqrt{\theta_0}} \quad  \text{for all} \quad \b_k <\bar\b.
\]

By iteratively applying Proposition \ref{p:2} a finite number of times, we obtain that there exists $K_0, \alpha_0$ large enough such that $\beta_k \geq \bar \beta$ at some iteration $k$. As before, since $x^* \in  \textup{supp}(f_k)$, thanks to Proposition \ref{p:1} we can conclude that there exists $K >K_0 , \alpha > \alpha_0$ large enough such that $\mathcal{V}(T^*) = \ve$ for some $T^* \in [0, K \Delta t]$.
\end{proof}
Some remarks are in order.
\begin{remark}^^>
\begin{itemize}
	\item
	By definition of the Wasserstein-2 distance and of $\mathcal{V}(t)$, it holds $W_2^2(f(t), \delta_{x^*}) \leq 2 \mathcal{V}(t)$. In particular we obtain that \[ W_2^2(f(T^*), \delta_{x^*}) \leq 2 \ve\,.\]
		\item 
In step 5 of Algorithm \ref{alg:mf} the  Fokker--Planck equation \eqref{eq:pde}
needs to be integrated from time $t_k$ to $t_k+\Delta t.$ Note that our proof is independent of the numerical method adopted. Below, we use the particle description in Algorithm \ref{alg:iter} due to the high-dimensionality of the problem but in principle one can adopt other approximations. 
	\end{itemize}
\end{remark}

\section{Numerical Examples}

The aim of this section is to show the performance of the proposed optimization method on benchmark problems and study how  different parameters influence the evolution of the underlying particle system. We also compare the results of the numerical simulations with the theoretical analysis of Section \ref{sec:3} and, in particular, we show that the  algorithm is capable to properly update the penalty parameter $\b$ without knowledge on $\bar\b.$ 

We start by simulating the mean-field particles dynamics on a two-dimensional problem, and compare the simulation  to  the theoretical analysis. We will then validate Algorithm \ref{alg:iter} on four different problems in dimension $d=5$ for different parameters settings. Finally, we test the algorithm on a high dimensional optimization problem, with dimension up to $d=20$. 
As suggested by the theoretical analysis and in related work^^>\cite{fhps20-2, carrillo2019consensus,benfenati2021binary}, large values of $\alpha$ lead to a faster convergence and to a higher success rate. For this reason, in all our tests we fix  $\alpha = 10^6$.

\subsection{Simulation of mean-field regime}
\label{sec:41}

For a numerical  simulation of  the mean-field equation we consider a particle based discretization with  a large number of particles, $N = 10^6$. The first problem we consider is given by  and  
\be
\begin{split}
&\min_{x\in \RR^2} j(x):= \frac12\sum_{i=1}^2 \left(\frac{x_i^4}5 - 2x_i^2 +x_i\right)  +10  \\
& \text{subject to} \quad g(x) = R(z) = \frac{1}{2}\sum_{i=1}^2 z_i^2 - 10\cos(2\pi z_i) + 5\leq 0
\end{split}
\label{pb:num1}
\ee
where $z$ performs a translation and a rotation of $x$, namely
\[ z = \begin{pmatrix}
\cos(\pi/6) & - \sin (\pi/6) \\
\sin(\pi/6) & \cos(\pi/6) 
\end{pmatrix} ( x - (1,1)^\top)\,.
\]

We note that $j(x)$ exhibits multiple local minima and a global minimum $\hat x$ which does not belong to the admissible set $\mathcal{M} = \{ x \in \RR^2 | g(x) \leq 0\}$. $\mathcal M$ constitutes of a set of disjoint discs and the problem \eqref{pb:num1} admits a unique global solution $x^*$, see Figure \ref{fig:visual1} for the problem visualization.

We consider the exact $\ell_1$-penalization  $r(x) = \|g(x)\|_1$. The function  $R$ is the Rastrigin function, which also presents several local minima and, as a consequence, the penalty function $P(x, \beta) = j(x) + \beta r(x)$ is  non convex for all values of $\beta$. Figure \ref{fig:visual2} shows, for instance, $P(x,\beta)$ for $\beta = 5$. Note that the  value of the true multiplier $\b$ is unknown but between $4$ and $5.$ 

\begin{figure}
\begin{subfigure}{0.48\linewidth}
\centering
\includegraphics[trim= 7cm 10.5cm 7cm 11cm, clip,width =0.98\linewidth]{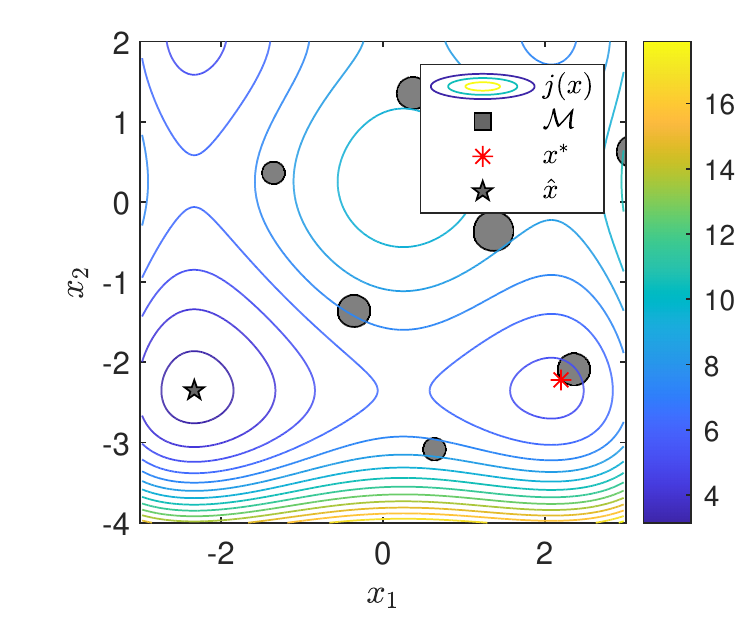}
\caption{Contour values of the objective function $j(x)$ and, in gray, the feasible set $\mathcal M$.}
\label{fig:visual1}
\end{subfigure}
\begin{subfigure}{0.48\linewidth}
\centering
\includegraphics[trim=  7cm 10.5cm 6.9cm 11cm, clip,width =1\linewidth]{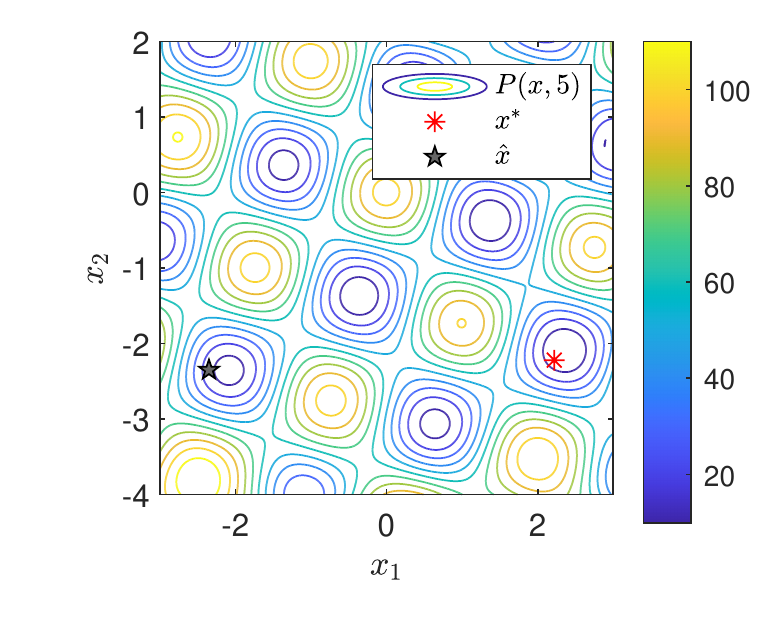}
\caption{Contour lines of the penalty function $P(x, \beta)$ for $\beta = 5$. }
\label{fig:visual2}
\end{subfigure}
\caption{A  representation of the constrained optimization problem \eqref{pb:num1} and its equivalent penalty subproblem. Both images also show the problem solution $x^*$ in red and the infeasible minimum $\hat x$ of $j(x)$ in black.}
\label{fig:visual}
\end{figure}

In order to compare with the theoretical analysis, we evolve $N = 10^6$ particles with isotropic diffusion, see \eqref{eq:iso}, and, to decide whether we need to update the parameter $\beta_k$, we compare the simple expectation of $r(x), \int r(x) df_k(x),$ with the tolerance $1/\sqrt{\theta_k}$. We set $\lambda = 1, \sigma= 0.5, \Delta t = 10^{-2}$ and the initial penalty parameter to $\beta_0 = 0.1$. Hence, initially  the penalty function is not(!) exact.

In the first simulation of the mean-field dynamics, we set the initial tolerance to $0.25$ $(\theta_0 = 16)$ and $\eta_\theta = 1.01$.
\cref{fig:evolution} shows the evolution of the particle density $f$ at different times while in \cref{fig:mfplota} we show the time evolution of the feasibility violation together with the tolerance $1/\sqrt{\theta_k}$ and $\beta_k$.  We note that, as long as $\beta_k$ is smaller than $\bar \beta$, $f$ concentrates around the infeasible minimizer $\hat x$, see \cref{fig:mfa,fig:mfb,fig:mfc}. At time $t=4$, $\beta_k > \bar \beta$ and $f$ spreads again, \cref{fig:mfd}, and then concentrates around the feasible solution $x^*$ afterwards, \cref{fig:mfe,fig:mff}.

\begin{figure}[h]
\begin{subfigure}{0.33\linewidth}
  \centering
  \includegraphics[trim= 7cm 11.2cm 8cm 11.7cm, clip, width=1\linewidth]{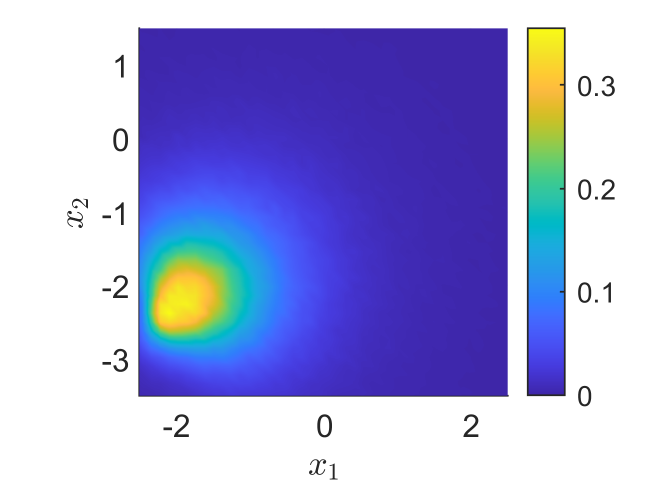}
  \caption{$t = 1, \beta/ \bar\beta \simeq 0.06 $}
 \label{fig:mfa}
\end{subfigure}%
\begin{subfigure}{0.33\linewidth}
  \centering
  \includegraphics[trim= 7cm 11.2cm 8cm 11.7cm, clip, width=1\linewidth]{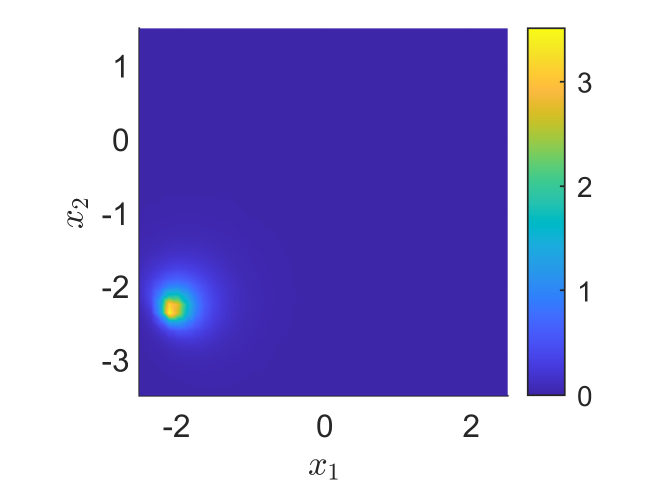}
  \caption{$t = 2, \b/ \bar \beta \simeq 0.17 $}
 \label{fig:mfb}
\end{subfigure}%
\begin{subfigure}{0.33\linewidth}
  \centering
  \includegraphics[trim= 7cm 11.2cm 8cm 11.7cm, clip, width=1\linewidth]{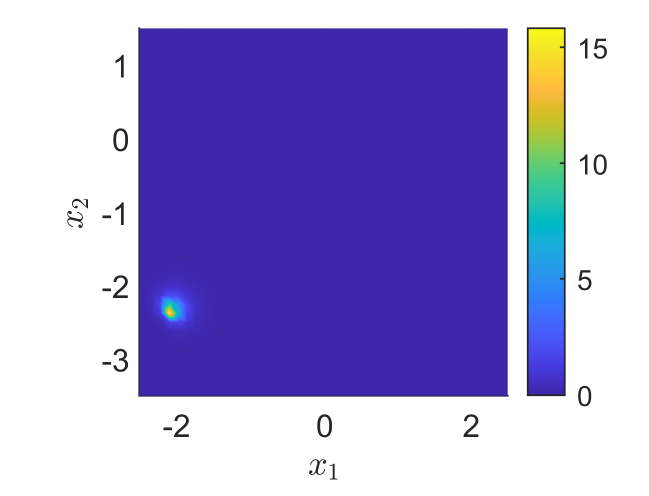}
 \caption{$t = 3, \b/\bar \beta \simeq 0.45 $}
 \label{fig:mfc}
\end{subfigure}
\begin{subfigure}{0.33\linewidth}
  \centering
  \includegraphics[trim= 7cm 11.2cm 8cm 11.7cm, clip, width=1\linewidth]{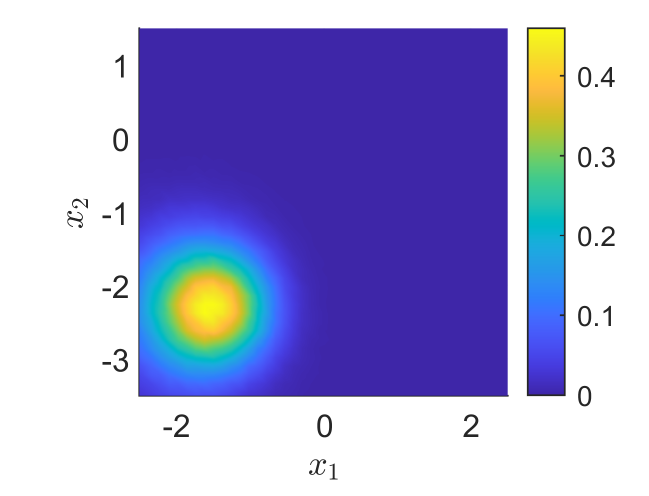}
  \caption{$t = 4, \b/\bar \b \simeq 1.22 $}
  \label{fig:mfd}
\end{subfigure}
\begin{subfigure}{0.33\linewidth}
  \centering
  \includegraphics[trim= 7cm 11.2cm 8cm 11.7cm, clip, width=1\linewidth]{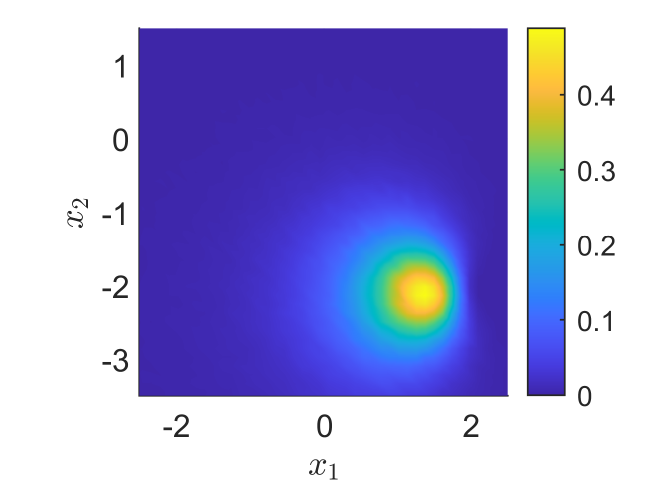}
  \caption{$t = 5, \b/\bar \b \simeq 3.31$}
 \label{fig:mfe}
\end{subfigure}
\begin{subfigure}{0.33\linewidth}
  \centering
  \includegraphics[trim= 7cm 11.2cm 8cm 11.7cm, clip, width=1\linewidth]{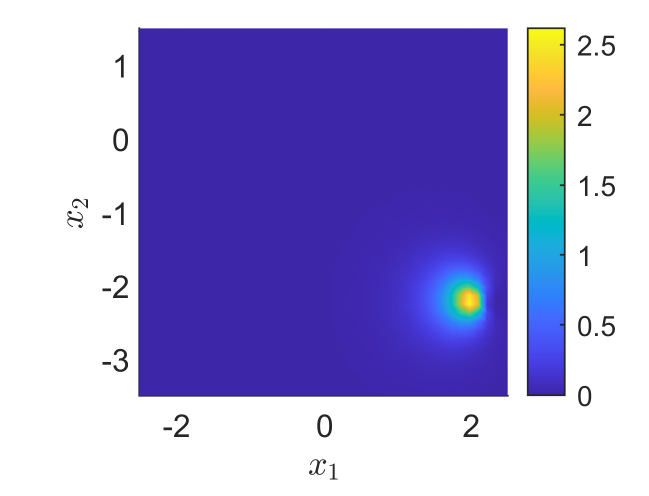}
  \caption{$t = 6, \b/\bar \b \simeq 8.96$}
 \label{fig:mff}
\end{subfigure}

\caption{Evolution of the particle density  at different times. Number of particles: $10^6$, Isotropic diffusion. $\lambda = 1, \sigma = 0.5, \Delta t = 0.01, \theta_0 = 2, \beta_0 = 0.1, \eta_\theta = 1.01, \eta_\b = 1.01$.}
\label{fig:evolution}
\end{figure}

For this particular choice of $\theta_0$, the constraint violation at the infeasible minimum $r(\hat x) \simeq 0.88$ is larger than the initial tolerance $0.25$, which means that condition \eqref{eq:infeas} of Proposition \ref{p:2} is satisfied. This not the case if we consider the case of $\theta_0 = 0.25$, as $1/\sqrt{\theta_0} = 2 > r(\hat x)$. Proposition \ref{p:2} suggests  that in this case the particle density might, concentrate around the infeasible minimizer, which is observed, see \cref{fig:mfplotb}.

In the first case,  we note that $\beta_k$ remains constant after a certain point (\cref{fig:mfplota}). This implies that the feasibility condition is satisfied and that the constraint violation decays faster than the tolerance. We recall that this behavior is expected due to   Proposition \ref{p32}. Therein, it has been shown that the  particles concentrate around the minimizer, see eq. \eqref{eq:p32}. Further,  if $\eta_\theta \leq \exp\left((\lambda - d\sigma^2/2) \Delta t \right)$, then the feasibility condition will be satisfied until the desired accuracy is reached. Here, the condition on $\eta_\theta$ is actually not satisfied as $\eta_\theta = 1.01$ is slightly larger than $\exp\left((\lambda - d\sigma^2/2) \Delta t \right) \simeq 1.008$. Nevertheless, in the simulation, the constraint violation decays faster than the tolerance. Using even larger values of $\eta_\theta$, like $\eta_\theta = 1.1$,   $\beta_k$ increases during the entire computation, see \cref{fig:mfplotc}.

To conclude, we note that if one computes the feasibility violation by using the Gibbs distribution, as in \eqref{eq:constr2}, the algorithm performance improves drastically. In particular, as seen in \cref{fig:mfplotd}, the feasibility condition is satisfied almost as soon as the threshold value $\bar \beta$ is reached. While in  the mean-field regime all shown methods lead to success, we will see in the next section that employing the Gibbs distribution within the feasibility check is essential to obtain results with strong performance also in the case of a {\em small } number of particles.

\begin{figure}[h!]
\begin{subfigure}{0.5\linewidth}
  \centering
  \includegraphics[trim= 6cm 10cm 6cm 10cm, clip, width=.95\linewidth]{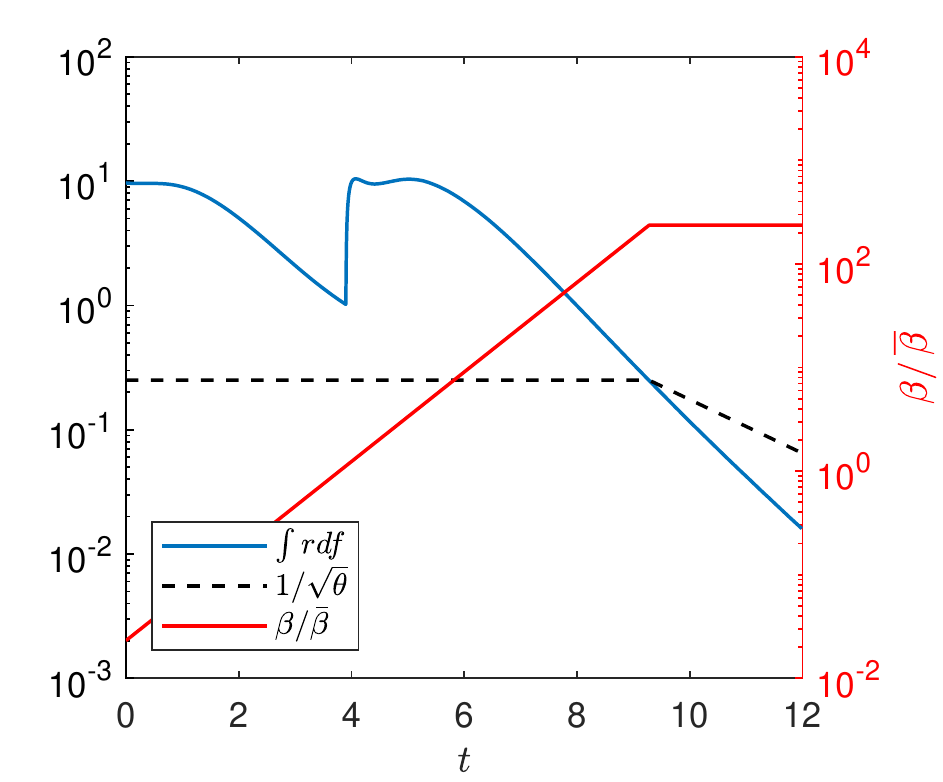}
  \caption{$1/\sqrt{\theta_0} = 0.25, \eta_\theta = 1.01.$}
 \label{fig:mfplota}
\end{subfigure}%
\begin{subfigure}{0.5\linewidth}
  \centering
  \includegraphics[trim= 6cm 10cm 6cm 10cm, clip, width=.95\linewidth]{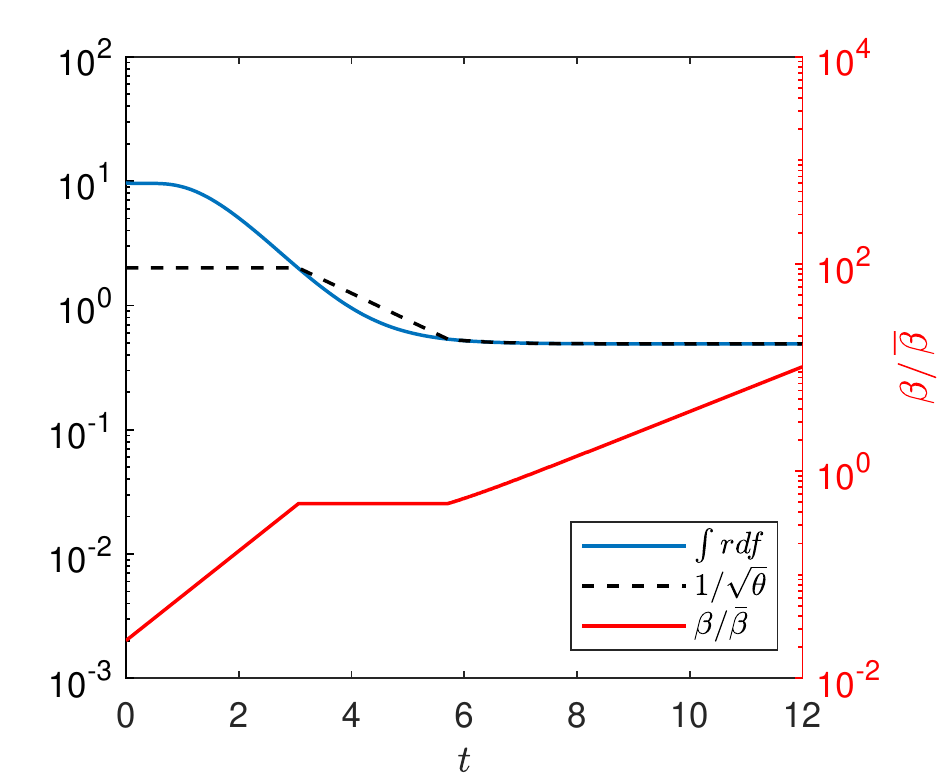}
  \caption{$1/\sqrt{\theta_0} = 2, \eta_\theta = 1.01.$}
 \label{fig:mfplotb}
\end{subfigure}%

\begin{subfigure}{0.5\linewidth}
  \centering
  \includegraphics[trim= 6cm 10cm 6cm 10cm, clip, width=.95\linewidth]{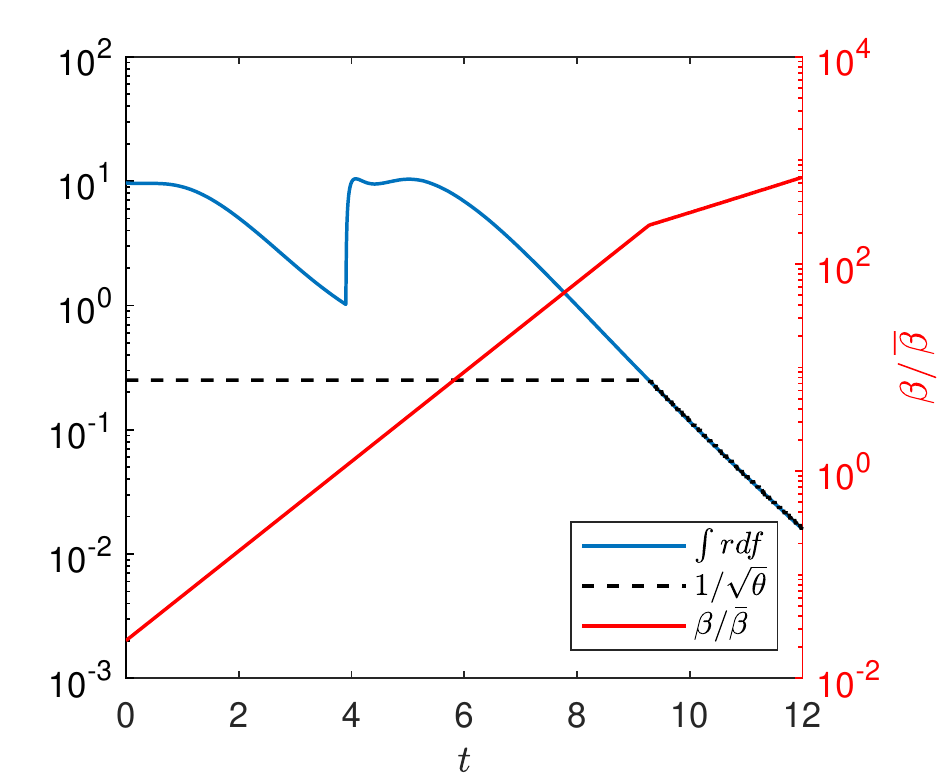}
 \caption{$1/\sqrt{\theta_0} = 0.25, \eta_\theta = 1.1.$}
 \label{fig:mfplotc}
\end{subfigure}
\begin{subfigure}{0.5\linewidth}
  \centering
  \includegraphics[trim= 6cm 10cm 6cm 10cm, clip, width=.95\linewidth]{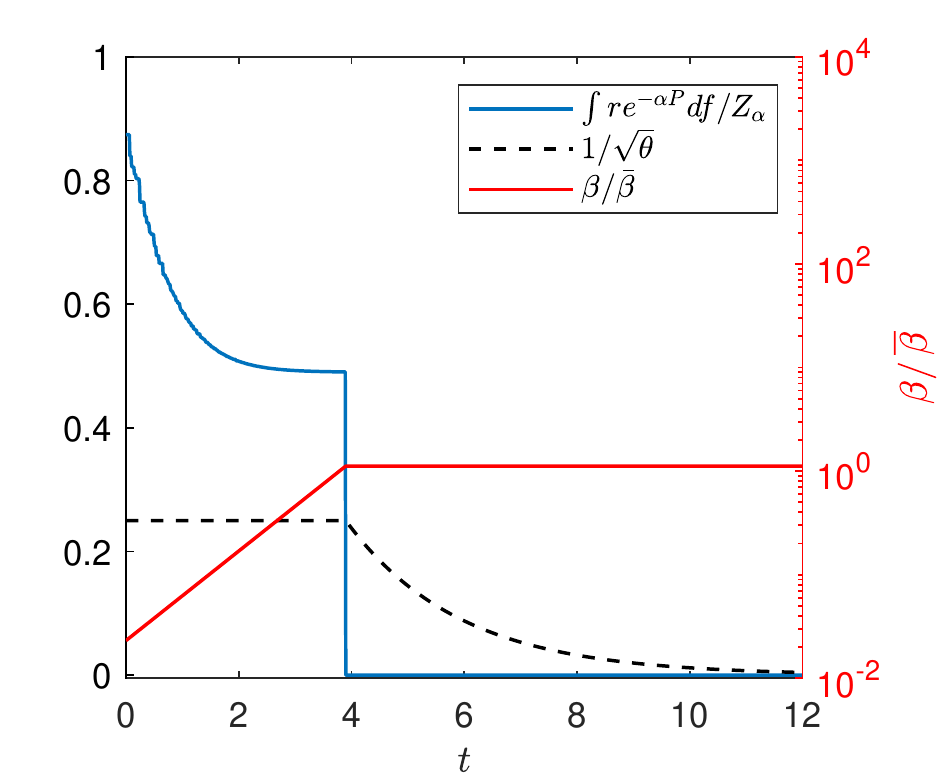}
 \caption{$1/\sqrt{\theta_0} = 0.25, \eta_\theta = 1.01,$ weighted feasibility check.}
 \label{fig:mfplotd}
\end{subfigure}
\caption{Evolution of the constraint violation (in blue), the tolerance $1/\sqrt{\theta_k}$ and the penalty parameter $\beta_k$ in the simulation of the mean-field algorithm applied to problem \eqref{pb:num1} in four different settings. Only in \cref{fig:mfplotd} the constraint violation is calculated using the Gibbs distribution as in \eqref{eq:constr2}}
\label{fig:mfplot}
\end{figure}

\subsection{Benchmark problems in $d=5$}

We validate the proposed algorithm by solving four different test problems for  objective functions 

\begin{align*}
j_1 (x) & = \frac1d\sum_{i=1}^d \left(\frac{x_i^4}5 - 2x_i^2 +x_i\right)  +10  
\\
j_2(x)  & = -20\exp\lp -0.2\sqrt{\frac{1}{d}\sum_{i=1}^{d}(x - o)_i^2}\rp-\exp\lp\frac{1}{d}\sum_{i=1}^{d}\cos(2\pi (x - o)_i)\rp+ 20 +e \\
&\notag \qquad \text{were} \quad o = (1.7\bar6, 1.5\bar3, 1.\bar3,1.0\bar 6, 0.8\bar 3) 
\end{align*}

with two admissible sets $\mathcal M_1 = \mathbb{S}^{4}, \mathcal M_2 = \mathbb{T}^{4}$, that is the sphere and the torus. 
The constrained optimization problems we consider are for $l,i=1,2$:
\be
\min_{x \in \RR^5} j_l(x) \quad \text{subject to} \quad x \in \mathcal{M}_i \quad \text{for}
\quad l=1,2\,, \quad i = 1,2\,.
\label{pb:num2}
\ee

As exact penalty, we use the distance function as in (A1.3):

\begin{align*}
r_1(x)  &= d(x, \mathbb{S}^4) = \left| \|x \| -1 \right|
\\
r_2(x) & = d(x, \mathbb{T}^4) = \left |\sqrt{ \left(\sqrt{\|x\|^2 - x_d^2} - 1\right)^2 + x_d^2}-0.5 \right |
\end{align*}
where $x_d$ is the $d$-component of $x$. We note that $j_1$ is the same objective function we used in problem \eqref{pb:num1}, while $j_2$ is the Ackley function. These objective functions have several local minima, both as functions on the whole domain $\RR^d$ and as functions restricted to the admissible sets $\mathcal M_1, \mathcal M_2$, making the optimization problems particularly challenging.

We will always run the algorithm several times with different initialization of $\beta_0$, since the true value $\bar\b$ such that (A1.1) holds is unknown. Specifically, we take $\beta_0$ in a range between $10^{-5}$ and $10^3$. For these test problems, $\bar \beta \in [1,10]$. 

We restrict ourselves  to the case where the particles evolve with the isotropic exploration process \eqref{eq:iso}, and we set $\sigma = 0.6$. We leave the comparison with the anisotropic process to the next section. The remaining parameters are set to be  $\lambda=1, \Delta t = 0.1, \theta_0 = 4, \eta_\theta = 1.1, \eta_\beta = 1.1.$ and we evolve $N = 200$ particles, initially sampled from a uniform distribution on $[-2,2]^d$, for $K = 300$ iterations. Finally, we consider a run to be successful if
\[ \| X_\alpha^K  - x^*\|_\infty \leq 0.1
 \]  
where $X_\alpha^k$ is defined as in \eqref{eq:Xa} and $x^*$ is the unique solution of the constrained problem we are solving.

\begin{figure}[t]
\begin{subfigure}{0.33\linewidth}
  \centering
  \includegraphics[trim= 7.5cm 10.5cm 7.5cm 10.5cm, clip, width=1\linewidth]{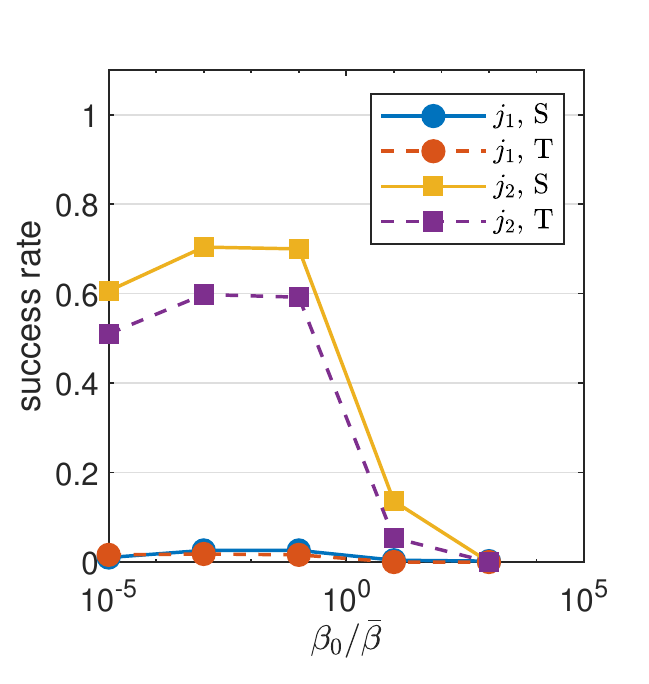}
  \caption{Feasibility check with \eqref{eq:expect}.}
 \label{fig:succ1}
\end{subfigure}%
\begin{subfigure}{0.33\linewidth}
  \centering
  \includegraphics[trim= 7.5cm 10.5cm 7.5cm 10.5cm, clip, width=1\linewidth]{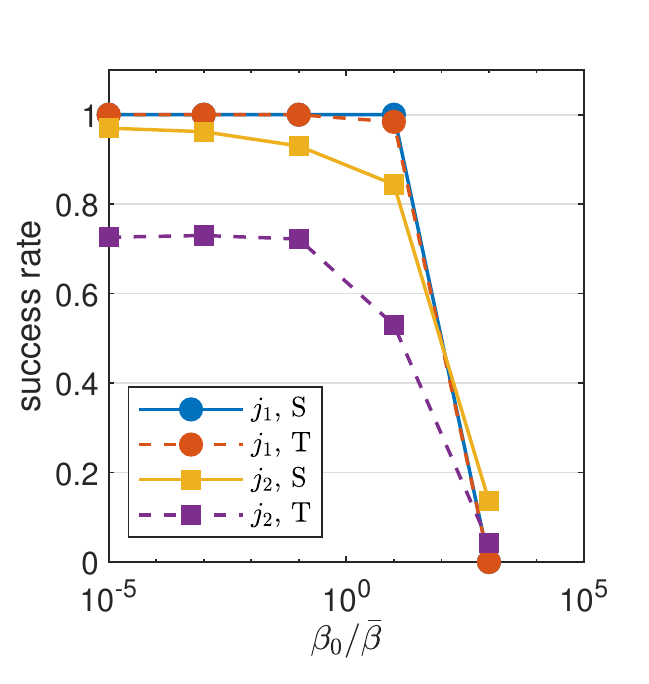}
  \caption{Feasibility check with \eqref{eq:gibbs}.}
 \label{fig:succ2}
\end{subfigure}%
\begin{subfigure}{0.33\linewidth}
  \centering
  \includegraphics[trim= 7.5cm 10.5cm 7.5cm 10.1cm, clip, width=1\linewidth]{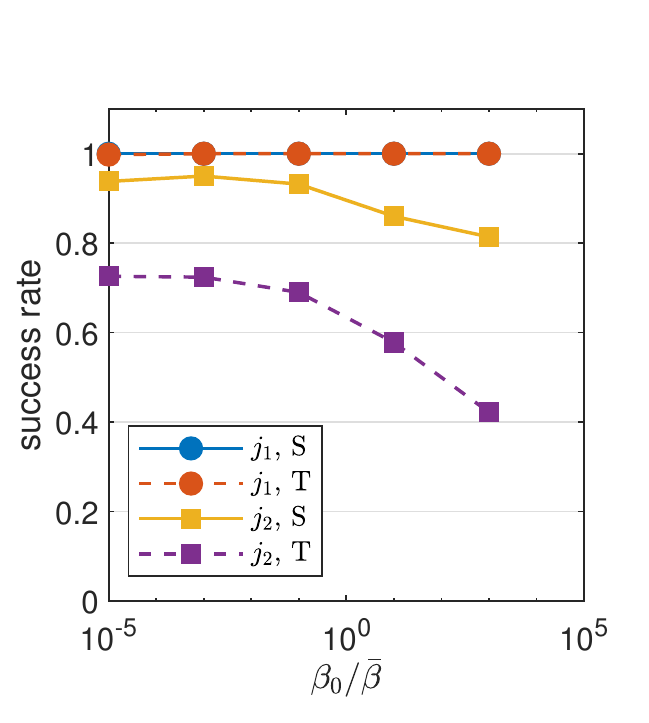}
  \caption{Feasibility check with \eqref{eq:gibbs} and decreasing strategy for $\beta$.}
 \label{fig:succ3}
\end{subfigure}%
\caption{Success rate obtained for different initialization of $\beta_0$. The tests are run 500 times to solve problems \eqref{pb:num2}. In (a) the feasibility check is performed with \eqref{eq:expect}, in (b) and (c) with \eqref{eq:gibbs}. In (c) the additional decreasing strategy for $\beta$ is also used.}
\label{fig:succ}
\end{figure}
\begin{figure}[h]
\begin{subfigure}{0.33\linewidth}
  \centering
  \includegraphics[trim= 7.5cm 10.5cm 7.5cm 10.5cm, clip, width=1\linewidth]{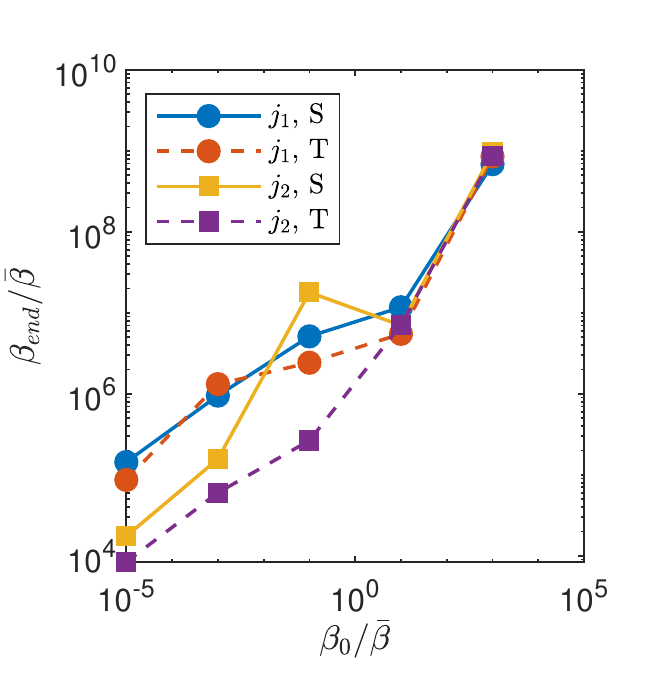}
  \caption{Feasibility check with \eqref{eq:expect}.}
 \label{fig:beta1}
\end{subfigure}%
\begin{subfigure}{0.33\linewidth}
  \centering
  \includegraphics[trim= 7.5cm 10.5cm 7.5cm 10.5cm, clip, width=1\linewidth]{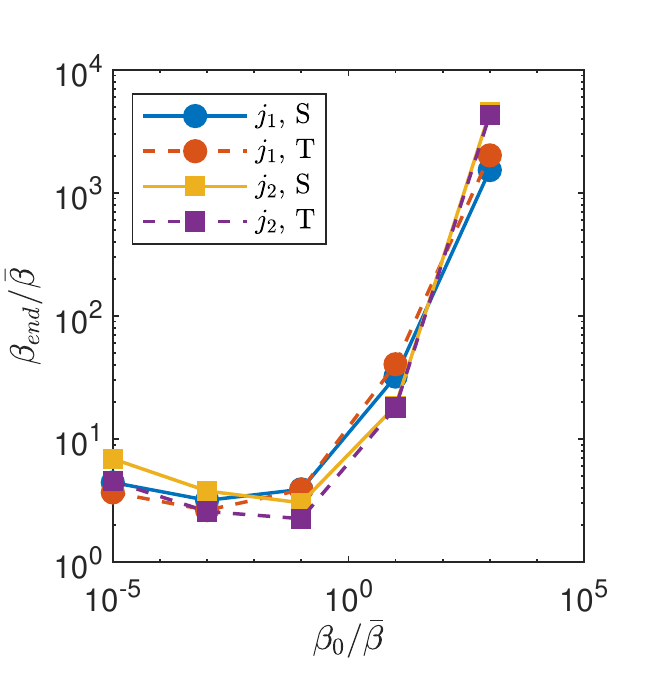}
  \caption{Feasibility check with \eqref{eq:gibbs}.}
 \label{fig:beta2}
\end{subfigure}%
\begin{subfigure}{0.33\linewidth}
  \centering
  \includegraphics[trim= 7.5cm 10.5cm 7.5cm 10.5cm, clip, width=1\linewidth]{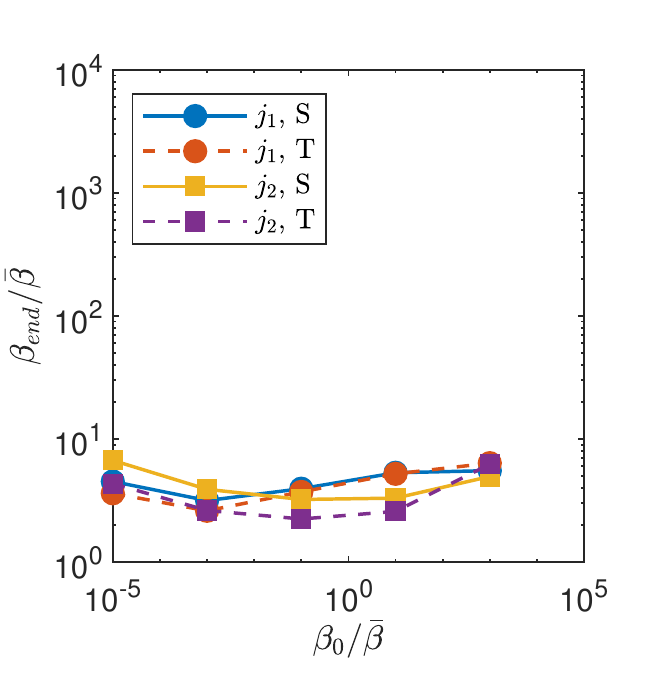}
  \caption{Feasibility check with \eqref{eq:gibbs} and decreasing strategy for $\beta$.}
 \label{fig:beta3}
\end{subfigure}%

\caption{Final value of $\beta$ after the computation as a function of the initial value $\beta_0$, both normalized with respect to $\bar \beta$. The tests are run 500 times to solve problems \eqref{pb:num2}. In (a) the feasibility check is performed with \eqref{eq:expect}, in (b) and (c) with \eqref{eq:gibbs}. In (c) the additional decreasing strategy for $\beta$ is also used.}
\label{fig:beta}
\end{figure}

We  validate the algorithm's performance on the feasibility check by computing the quantity \eqref{eq:constr1}: 
\be
\frac 1N \sum_{i=1}^N r(X_k^i)\, . 
\label{eq:expect}
\ee
As seen in Fig. \cref{fig:succ1}, the success rate is rather poor for almost all problems and for several values of $\beta_0$. This is due to the fact that,  the feasibility check is violated several times and, as a consequence, the penalty parameter increases, reaching large values at the end of the computation, see \cref{fig:beta1}.  

As we mentioned in Section \ref{sec:41}, considering the weighted expectation of $r$
\be
\frac{1}{Z_\alpha}\sum_{i=1}^N\, r(X_k^i)\, \exp\left(-\alpha \pen(X_k^i,\beta)\right),  
\label{eq:gibbs}
\ee
instead of \eqref{eq:expect}, improves the algorithm performance. In particular, if $\beta_0$ is not too large, we obtain a success rate close to one when the objective function is $j_1$, as we can see from \cref{fig:succ2}. \cref{fig:beta2} shows, indeed, that the final value of the penalty parameter does not overshoot $\bar \beta$ significantly. 

In all the problems, we observe that if $\beta_0 > \bar \beta$, $\beta_k$ increases moderately during the computation, which is a desired feature of the algorithm (\cref{fig:beta2}). This is not enough, though, to successfully solve the problem if our initial guess of the penalty parameter is too large, that is for instance when we choose $\beta_0 = 10^3$ for these test problems, see again \cref{fig:succ2}. 
As an interpretation, we may argue that this happens, since  the methods tries to minimize the penalty function $P(x) = j(x) + \beta r(x)$ and the penalty term $\b r(x)$ overwhelms $j(x)$ if $\beta \gg \bar \beta$.
\par 
To overcome this issue, we propose the following heuristic strategy: during the computation we decrease $\beta$ by setting $\beta_{k+1} = \beta_k/ \eta_\beta$  until the feasibility condition is violated for the first time. Applying this simple strategy, the algorithm performance becomes less sensitive to the choice of $\beta_0$, see \cref{fig:succ3,fig:beta3}.

We conclude by remarking that when optimizing $j_2$ (the Ackley function) either over $\mathbb{S}^4$ or $\mathbb{T}^4$, the proposed algorithm is not able to reach a success rate of $1$ (\cref{fig:succ3}), although it is able to properly adapt the penalty parameter (\cref{fig:beta3}). This is due the fact that, in both cases, the particles may  concentrate on a local minimum which satisfies the constraint and it is very close to the  true solution in terms of the value of objective function $j_2$.

\subsection{Benchmark problems in higher dimensions}
\label{s:43}

Being able to tackle high dimension problems is of paramount importance in applications and particles methods seem to be able to perform well even when $d\gg 1$,^^>\cite{benfenati2021binary, carrillo2019consensus,fhps20-2}. We recall that, 
to obtain convergence guarantees for the CBO method with isotropic diffusion, the drift and the diffusion parameters $\lambda$ and $\sigma$ need to satisfy the condition $2\lambda> d\sigma^2$, which is dimension dependent (see \cref{t:conv}). This makes the parameters choice particularly restrictive for high dimensional problems. 
To overcome this issue, the CBO method with anisotropic exploration \eqref{eq:aniso} has been introduced in^^>\cite{carrillo2019consensus}. Here, the noise is added to the particle dynamics component-wise and, as a consequence, the restriction between the parameters $\lambda$ and $\sigma$ becomes $2\lambda>\sigma^2$. We refer to^^>\cite{carrillo2019consensus, fornasier2021convergence, fhps20-3}  for more details.

To compare the use of isotropic and anisotropic explorations in Algorithm \ref{alg:iter}, we test the methods on the following scalable constrained optimization problems:

\be
\label{pb:qp}
\min_{x \in \RR^d}\, \frac12 x^\top A x - b^\top x \quad \text{subject to}
\quad H^\top x + h^0 = 0, \; x\geq 0\,,
\ee
where $A \in \RR^{d \times d}$ is a symmetric positive definite matrix
and $H\in \RR^{p \times d}$, $p=\lfloor d/2 \rfloor$. 
The interested reader can find the details on how to randomly generate \eqref{pb:qp} in^^>\cite{spellucci2002}. We note that the problem is quadratic, convex and it admits a unique global solution $x^*$. We use $\ell_1$-penalization, $r(x) = \| H^\top x - h^0\|_1$, which is exact, and generate the problem such that $\bar \beta$ is approximately $1$, hence (A1.1) if fulfilled. In the tests, we consider $d=10, 15, 20$, fix $\lambda =1$, while we vary the diffusion parameter $\sigma$. The remaining parameters are set to $\Delta t = 0.1, \theta_0 = 4, \eta_\theta = 1.05, \eta_\beta = 1.05, K = 300, N = 500$, that is we use $500$ particles. A run is considered successful when $ \| X_\alpha^K  - x^*\|_\infty \leq 0.25$.  

As in the case of unconstrained optimization, the proposed algorithm performs well for limited values of $\sigma$ when using isotropic diffusion, see \cref{fig:qp1}. In particular, the optimal value changes according to $d$, showing the dependence on the  dimension of the problem. Anisotropic diffusion, instead, allows both to reach a higher success rate and to obtain computational results also for  a wider range of values of $\sigma$, as shown in \cref{fig:qp2}.

\begin{figure}[t]
\begin{subfigure}{0.5\linewidth}
  \centering
  \includegraphics[trim=6cm 10.5cm 6.5cm 11.2cm, clip, width=0.9\linewidth]{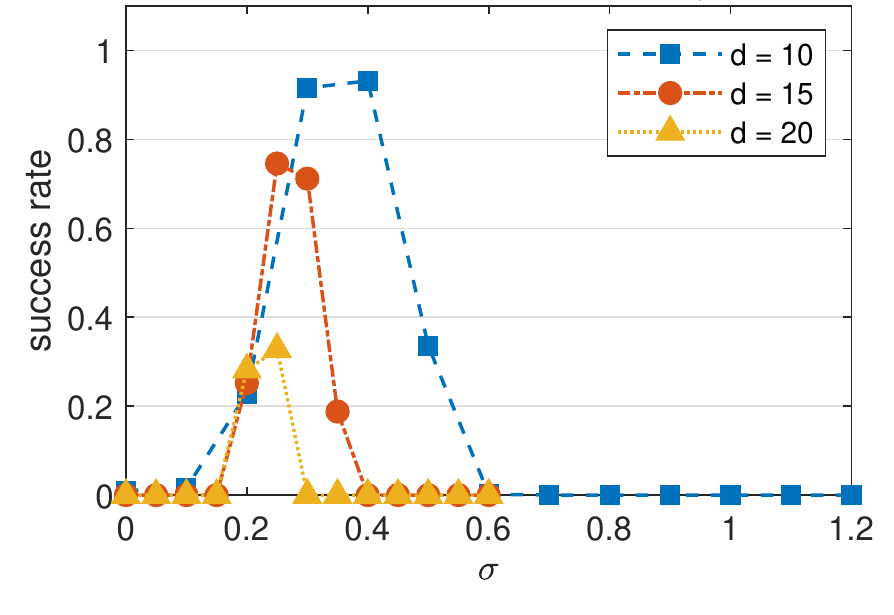}
  \caption{Isotropic exploration.}
 \label{fig:qp1}
\end{subfigure}%
\begin{subfigure}{0.5\linewidth}
  \centering
  \includegraphics[trim= 6cm 10.5cm 6.5cm 11.2cm, clip, width=0.9\linewidth]{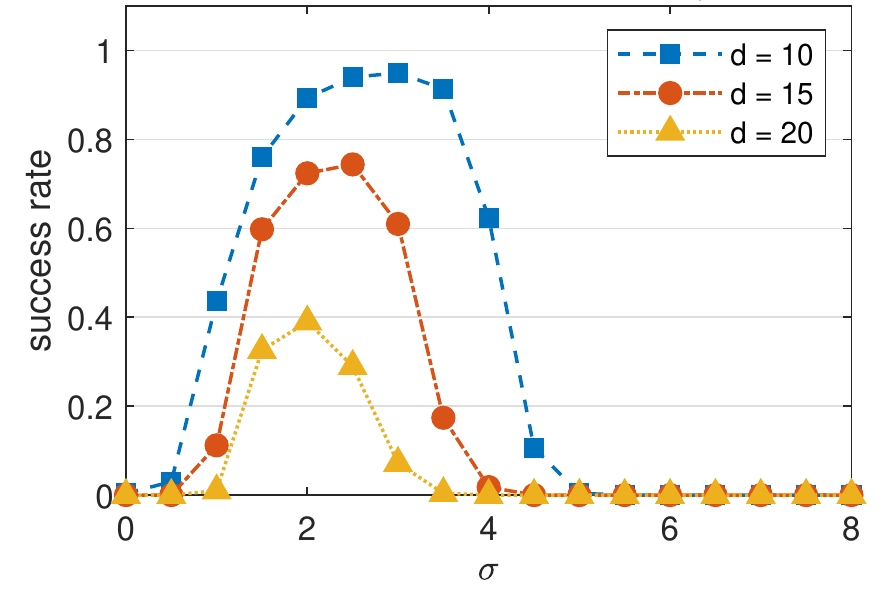}  
  \caption{Anisotropic exploration.}
 \label{fig:qp2}
\end{subfigure}%
\caption{Success rate as a function of the diffusion parameter $\sigma$ for the two different exploration processes. The results are averaged over 500 runs of the algorithm, used to solve \eqref{pb:qp}.}
\label{fig:qp}
\end{figure}

To conclude, we mention some random batch techniques that are typically used to speed-up the converge of particle-based methods and that can also be applied within  Algorithm \ref{alg:iter}. The first consists of the approach introduced in^^>\cite{AlPa}, where a random subset $\{X_k^j\,;\, j \in J_M\}$, of $|J_M|=M<N$ particles is selected at each iteration. Then $X_k^\alpha$ is calculated within the batch as
\[ X_k^\alpha =  \frac{1}{Z_{\alpha,M}}\sum_{j \in J_M} X_k^j \exp\left(-\alpha \pen(X_k^j,\beta)\right), \quad Z_{\alpha,M} = \sum_{j \in J_M}\exp\left(-\alpha \pen(X_k^j,\beta)\right)\,.
\]
One may then decide to update the entire set of particles or the batch only, reducing in this way the computational cost per iteration from $\mathcal{O}(N)$ to $\mathcal{O}(M)$. Another batch approach consists of dividing all the particles in $S$ different batches $J_M^1, \dots, J_M^S$ with $SM = N$ and let them interact only within the assigned batch, see^^>\cite{carrillo2019consensus,JLJ } for more details. We remark that these methods not only save computational time but also add additional stochasticity to the particle dynamics which might improve the algorithm's performance.

\section{Conclusions}
In this work we have extended the class of CBO methods, a novel class of gradient-free methods recently introduced in the context of global optimization of nonconvex functionals in high dimension^^>\cite{pinnau2017consensus,carrillo2019consensus}, to the case of constrained minimization problems. To this end, we introduced a penalty term in the constrained problem and derived an iterative procedure to determine the optimal penalty parameter based on the constraint violation  by the particle system.  
In particular, following the strategy based on analyzing the system behavior for a large number of particles via the corresponding mean-field limit^^>\cite{fornasier2021consensusbased}, we then proved convergence to the constrained minimum for a large class of nonlinear problems. Even if the mathematical analysis is carried on for isotropic exploration processes, extension to the anisotropic case are discussed in view of the recent result in^^>\cite{fornasier2021convergence}. 
The theoretical analysis is confirmed by numerical simulations of the system behavior in the mean-field limit. Numerous applications to constrained minimization problems in high dimension are also presented showing the very good performance of the new numerical method.

\subsection*{Acknowledgments}
This work has been written within the
activities of GNCS group of INdAM (National Institute of
High Mathematics). L.P. acknowledge the partial support of MIUR-PRIN Project 2017, No. 2017KKJP4X “Innovative numerical methods for evolutionary partial differential equations and applications”.  
The work of G.B. is funded by the Deutsche Forschungsgemeinschaft (DFG, German Research Foundation) – Projektnummer 320021702/GRK2326 – Energy, Entropy, and Dissipative Dynamics (EDDy).
%(Herty;s email) The authors thank the Deutsche Forschungsgemeinschaft (DFG, German Research Foundation) for the financial support through 20021702/GRK2326, 333849990/IRTG-2379, HE5386/19-2,22-1,23-1 and under Germany's Excellence Strategy EXC-2023 Internet of Production 390621612. 
%%%%

\bibliographystyle{abbrv}
\bibliography{Constrained-swarm-opt-v3}

\end{document}